\documentclass[12pt]{amsart}
\usepackage[margin=2.5cm]{geometry}

\usepackage{graphicx}%
\usepackage{multirow}%
\usepackage{amsmath,amssymb,amsfonts}%
\usepackage{mathrsfs}%
\usepackage[title]{appendix}%
\usepackage{xcolor}%
\usepackage{textcomp}%
\usepackage{manyfoot}%
\usepackage{booktabs}%
\usepackage{algorithm}%
\usepackage{algorithmicx}%
\usepackage{algpseudocode}%
\usepackage{listings}%
\usepackage{latexsym,amsmath,amssymb,amsthm,amscd,mathtools,tikz,microtype,comment,color}
\usepackage[english]{babel}
\usepackage{enumitem,MnSymbol}
\usepackage{multicol}
\usepackage[lofdepth,lotdepth]{subfig}

\usepackage{multicol}
\usepackage{faktor, xfrac,nicefrac}

\DeclareMathOperator{\codim}{codim}

\newcommand{\K}{{\mathbb{K}}}
\newcommand{\N}{{\mathbb{N}}}

\newcommand{\ol}{\overline}

\theoremstyle{definition}
\newtheorem{theorem}{Theorem}[section]%
\newtheorem{proposition}[theorem]{Proposition}%
\newtheorem{lemma}[theorem]{Lemma}

\newtheorem{corollary}[theorem]{Corollary}

\newtheorem{question}[theorem]{Question}

\newtheorem{example}[theorem]{Example}%
\newtheorem{remark}[theorem]{Remark}%

\newtheorem{definition}[theorem]{Definition}%

\begin{document}

\title[Comparability of the total Betti numbers of toric ideals of graphs]{Comparability of the total Betti numbers of toric ideals of graphs}

\author{Giuseppe Favacchio}
  \address[G. Favacchio]{Dipartimento di Ingegneria, Università degli studi di Palermo, Viale delle Scienze, Palermo, 90128, Italy}
 \email{giuseppe.favacchio@unipa.it}




\thanks{\noindent{\bf Acknowledgements.}   The author  is member of the National Group for Algebraic and Geometrical Structures and their Applications (GNSAGA-INdAM). 
The results were inspired by
calculations using {\em CoCoA} \cite{cocoa}.}

\maketitle

\begin{abstract}
The total Betti numbers of the toric ideal of a simple graph are, in general, highly sensitive to any small change of the graph. In this paper we look at some combinatorial operations  that cause total Betti numbers to change in predictable ways. In particular, we focus on a procedure that preserves these invariants.
\end{abstract}

\section{Introduction}\label{s.intro}

Given a finite simple graph $G$ on the vertex
set $V = \{v_1,\ldots,v_n\}$ with edge set $E = \{e_1,\ldots,e_q\}$,
the {\it toric ideal of $G$}, denoted $I_G$, is the 
kernel of the map $\varphi : \K[E] = \K[e_1,\ldots,e_q] \rightarrow \K[v_1,\ldots,v_n]$ given by $\varphi(e_i) = v_{i_1}v_{i_2}$ where $e_i = \{v_{i_1},v_{i_2}\} \in E$.  
The standard graded $\K$-algebra $\K[E]/I_G$ will be denoted by $\K[G]$.

Properties of toric ideals of graphs have been extensively studied in recent years from several points of view, see \cite{ ballard2021, constantinescu2018,Adam2022,erey2021,FHKVT,FKVT20, robustness2022,ha2019algebraic,hibi2020, K-thesis,nandi2019,neves2022} just to cite some of them. Many questions remain unanswered, for instance it is not yet known a characterization of the graphs $G$ such that the algebra $\K[G]$ is Cohen-Macaulay, see Section~\ref{ss.hom inv} for the definition. More generally, it is important to explore the procedures which transform a graph $G$ into a new graph $G'$ such that the Betti numbers (see Section \ref{ss.hom inv}) of $K[G']$ are predictable from those of $K[G]$. 

We start by recalling the relevant background on the homological invariants, graphs and their toric ideals.
\subsection{Notation: Homological Invariants}\label{ss.hom inv}  Given an homogeneous ideal  $I$ in a polynomial ring $R$,  
the minimal graded free resolution of $R/I$ has the form
\[0 \rightarrow \bigoplus_{j \in \mathbb{N}} R(-j)^{\beta_{p,j}(R/I)}
\rightarrow \cdots \rightarrow \bigoplus_{j \in \mathbb{N}} R(-j)^{\beta_{1,j}(R/I)}
\rightarrow R \rightarrow R/I \rightarrow 0\]
where we recall that $R(-j)$ denotes the ring $R$ with its grading shifted by
$j$, and $\beta_{i,j}(R/I)=\dim_\K {\rm Tor}_i^R(R/I, \K)_j$ is called the ${i,j}$-th {\em graded Betti number} of $R/I$. The numbers $\beta_i(R/I)=\sum_j\beta_{i,j}(R/I)$ are called {\it total Betti numbers} of $R/I$. The algebra $R/I$ is {\it Cohen-Macaulay} (CM for short) if and only if its depth equals its (Krull) dimension.  

In this paper we are investigate some operations on graphs which have a {\it good behavior} with respect the total Betti numbers and the Cohen-Macaulay property.
It is known that, see \cite[Corollary 3.5]{ha2019algebraic}, if $H$ is a subgraph of $G$ then $\beta_{ij}(\K[H])\le \beta_{ij}(\K[G])$.  However, as noticed for instance in \cite[Example 5.3]{ha2019algebraic}, this inequality does not ensure the persistence of the Cohen-Macaulay property from $\K[G]$ to $\K[H]$.  Indeed, such phenomenon could happen if $\codim \K[H]< \codim \K[G]$.

\subsection{Notation: Graphs} 
Let $G=(V,E)$ be a simple graph and $x\in V$. We denote by $\mathcal N_G(x)=\{y \in V\ |\ \{x,y\}\in E \}$, the {\it neighbors set} of  $x$. We denote by $d_G(x)=|N_G(x)|.$

A {\it  walk} in a graph $G=(V,E)$  is a sequence of vertices $w=(x_0,x_1,x_2,\ldots, x_n)$  where $\{x_i,x_{i+1}\}\in E$ for $i= 0,\ldots, n-1$.   The {\it length} of the walk $w=(x_0,x_1,x_2,\ldots, x_n)$ is $n$ and it is denoted by $|w|$. 
 A  walk $w$  is even (odd) if $n$ is even (odd), i.e., it consists of an even (odd) number of edges.  A walk is closed if $x_0=x_n$.
Given a walk $w=(x_0,x_1,x_2,\ldots, x_n)$ it will be useful to work with the edges $e_i=\{x_{i-1},x_i\}$, so we will use the same terminology and we will also call {\it walk} the sequence $(e_1, \ldots e_n)$. With an abuse of notation we write $w=(e_1, \ldots e_n)=(x_0,x_1,x_2,\ldots, x_n)$. 

Let $w,w'$ be two walks in $G$, if the last vertex in $w$ is the fist vertex in $w'$, then the symbol  $w||w'$ denotes the concatenation of the two walks.  

We  say that a walk $p=(x_0,x_1,x_2,\ldots, x_n)$ is a {\it  path} if $d_G(x_i)=2$ for $i=1,\ldots, n-1,$ i.e. the vertices $x_1,\ldots, x_{n-1}$ don’t have neighbors outside of the walk.   
For a path $p=(x_0,x_1,\ldots , x_{t})$ we denote by $p^r$ the {\it reverse path} of $p$, that is $p^r=(x_{t}, \ldots,x_1, x_0)$. 

\subsection{Notation: Toric Ideals of Graphs}   We refer the reader to \cite[Section 5.3]{herzog2018binomial} for a more exhaustive overview of the topic.  Let $G=(V,E)$ be a simple graph. 
Given a  list of edges of $G$, say $w=(e_{{1}},e_{{2}},\dots, e_{{n}})$, we denote by~$e_w$ the monomial given by the product of all the edges in $w$, i.e., 
\[e_w=\prod\limits_{j} e_{j} \in \K[E].\]

We associate to an even list of edges $w=(e_{{1}},e_{{2}},\dots, e_{{2n}})$   two sub-lists which consist of the odd and even entries \[w^+=(e_{{1}},e_{{3}},\dots, e_{{2n-1}})\ \ \text{and}\ \ w^-=(e_{{2}},e_{{4}},\dots, e_{{2n}}).\]
Moreover, we set $f_{w}\in \K[W]$ to be the homogeneous binomial defined by 
\[f_w=e_{w^+} \ - \  e_{w^-}.\]

A set of generators of the toric ideal $I_G$, (which also constitutes a universal Gr\"obner basis  of $I_{G}$, see \cite[Proposition 10.1.10]{V}),  corresponds to the primitive closed even walks in $G$. Recall that a closed even walk $w$ in a graph $G$ (and the correspondent binomial $f_w$) is said to be   {\it primitive} if there is not another closed even walk $v$ in $G$ such that $e_{v^+} \mid_{\ e_{w^+}}$ and $e_{v^-} \mid_{\ e_{w^-}}$. 

\subsection{Results of the paper} 
We now summarize the content of the paper. 

In Section \ref{s. even path contraction} we recall the definition of path contractions in a graph and we introduce the simple path contractions (see Definition \ref{d. simple contraction}). We prove some preliminary lemmas, and we enunciate the following two results which provide a relation between the path contraction and the total Betti numbers of  graph. 

\noindent {\bf Theorem \ref{t. main 0}} {\it Let $G$ be a simple graph. Let $G/p$ be a even path contraction of $G$.   Then $$\beta_i(\K[G])\ge \beta_i(\K[G/p])\ \  for\ any\ i \ge 0.$$}

\noindent {\bf Theorem \ref{t. main 1}} {\it Let $G$ be a simple graph. Let $G/p$ be a even simple path contraction of $G$.  If  $p\subset q$ where $q$ is a path in $G$ such that $|q|\ge |p|+2$,  then
$$\beta_i(\K[G])= \beta_i(\K[G/p])\ \ for\ any\ i \ge 0.$$}
We derive several corollaries and we discuss significant examples related to these results.
The proof of Theorem \ref{t. main 1} is postponed to Section \ref{s. proof of main1} where the background knowledge on simplicial complexes is introduced.

In Section \ref{s.edge contraction} we begin  the study of edge contractions for a special class of graphs introduced in Definition \ref{d. connected by the edge e}. The main result of this section is the following.

\noindent {\bf Theorem \ref{t.main 2}} {\it
 Let $G=G_1\stackrel{e}{-}G_2$ be a graph connected by the edge $e$. Then $$\beta_1(\K[G])=\beta_1(\K[G/e]).$$}

\section{Path contractions in a graph and the total Betti numbers}\label{s. even path contraction}
We recall the well known definition from Graph Theory of contraction of an edge in a~graph.

\begin{definition}\label{d. edge contraction}
Let $G=(V,E)$ be a simple graph. Let $e=\{x_0,x_1\}\in E$ and $y_0$ a new vertex. Set $V'=(V\cup\{y_0\})\setminus e$. Let $\chi:V\to V'$ be the function which maps every vertex in $V\setminus e$ into itself and both $x_0$ and $ x_1$ into $y_0$.
Consider the set $E'=\{ \{\chi(x), \chi(x')\} \ | \ \{x,x' \}\in E\setminus\{e\} \}$, where eventual redundancies are ignored.  
The graph $(V',E')$, denoted by $G/e$,  is said to be an {\it edge contraction} of $G$. 

Let $p$ be a path in $G$. A {\it path contraction} of $G$, denoted by $G/p$, is the graph obtained by contracting all the edges in $p$.
We say that a path contraction is {\it even (odd)} if $p$ is an even (odd) path.
\end{definition}
The graph $G/p=(V',E')$ in the above definition is simple.

We  also need the following definition.
\begin{definition}\label{d. simple contraction}
Let $G=(V,E)$ be a simple graph. Let $p$ be a path in $G$ and $x$ and $x'$ be the first and the last vertices in $p$. We say that a path contraction is {\it simple} if $\mathcal N(x) \cap \mathcal N(x')\subseteq p$.
\end{definition}
In particular, a path of length 2, $p=(x,y,x'),$ is simple when $\mathcal N(x) \cap \mathcal N(x')=\{y\}$ and a path of length greater than 2  is simple when $\mathcal N(x) \cap \mathcal N(x')=\emptyset.$

\begin{example}
The walks in bold in Figure \eqref{fig.1} are not paths. 
 \begin{figure}[ht]
	\centering
	\subfloat[ ]{
		\begin{tikzpicture}[scale=0.4]
		\coordinate (v1) at (0,0);
		\coordinate (v2) at (-2,-2);
		\coordinate (v3) at (-2,2);
		\coordinate (v7) at (2,-2);
		\coordinate (v9) at (2,2);	
		\coordinate (v8) at (4,-2);
		\coordinate (v10) at (4,2);
		\coordinate (v4) at (6,0);
		\coordinate (v5) at (8,-2);	
		\coordinate (v6) at (8,2);
		\draw[line width=2] (v1)--(v2);
		\draw (v2)--(v3);
		\draw (v3)--(v1);
		\draw (v1)--(v7);
		\draw (v7)--(v8);
		\draw (v8)--(v4);
		\draw[line width=2]  (v4)--(v10);	
		\draw[line width=2] (v10)--(v9);
		\draw[line width=2] (v9)--(v1);	
		\fill[fill=white,draw =black] (v1) circle (0.3);
		\fill[fill=white,draw =black] (v2) circle (0.3);
		\fill[fill=white,draw =black] (v3) circle (0.3);	\fill[fill=white,draw =black] (v4) circle (0.3);
		\fill[fill=white,draw =black] (v7) circle (0.3);
		\fill[fill=white,draw =black] (v8) circle (0.3);
		\fill[fill=white,draw =black] (v9) circle (0.3);
		\fill[fill=white,draw =black] (v10) circle (0.3);
		\end{tikzpicture}
		\label{fig.1A}}		\qquad
 	\subfloat[ ]{
		\begin{tikzpicture}[scale=0.4]
		\coordinate (v1) at (0,0);
		\coordinate (v2) at (-2,-2);
		\coordinate (v3) at (-2,2);
		\coordinate (v7) at (2,-2);
		\coordinate (v9) at (2,2);	
		\coordinate (v8) at (4,-2);
		\coordinate (v10) at (4,2);
		\coordinate (v4) at (6,0);
		\coordinate (v5) at (8,-2);	
		\coordinate (v6) at (8,2);
		\draw (v1)--(v2);
		\draw (v2)--(v3);
		\draw (v3)--(v1);
		\draw[line width=2] (v1)--(v7);
		\draw  (v7)--(v8);
		\draw  (v8)--(v4);
		\draw   (v4)--(v10);	
		\draw  (v10)--(v9);
		\draw[line width=2] (v9)--(v1);	
		\fill[fill=white,draw =black] (v1) circle (0.3);
		\fill[fill=white,draw =black] (v2) circle (0.3);
		\fill[fill=white,draw =black] (v3) circle (0.3);	\fill[fill=white,draw =black] (v4) circle (0.3);
		\fill[fill=white,draw =black] (v7) circle (0.3);
		\fill[fill=white,draw =black] (v8) circle (0.3);
		\fill[fill=white,draw =black] (v9) circle (0.3);
		\fill[fill=white,draw =black] (v10) circle (0.3);
		\end{tikzpicture}
		\label{fig.1B}}		\qquad
 \caption{}
	\label{fig.1}
\end{figure}

The walks in bold in Figure \eqref{fig.2} are paths. In particular, the one in Figure \eqref{fig.2A} is simple and the one in Figure \eqref{fig.2B} is not simple.

\begin{figure}[ht]
	\centering
	\subfloat[ ]{
		\begin{tikzpicture}[scale=0.4]
		\coordinate (v1) at (0,0);
		\coordinate (v2) at (-2,-2);
		\coordinate (v3) at (-2,2);
		\coordinate (v7) at (2,-2);
		\coordinate (v9) at (2,2);	
		\coordinate (v8) at (4,-2);
		\coordinate (v10) at (4,2);
		\coordinate (v4) at (6,0);
		\coordinate (v5) at (8,-2);	
		\coordinate (v6) at (8,2);
		\draw (v1)--(v2);
		\draw (v2)--(v3);
		\draw (v3)--(v1);
		\draw (v1)--(v7);
		\draw (v7)--(v8);
		\draw (v8)--(v4);
		\draw[line width=2]  (v4)--(v10);	
		\draw[line width=2] (v10)--(v9);
		\draw[line width=2] (v9)--(v1);	
		\fill[fill=white,draw =black] (v1) circle (0.3);
		\fill[fill=white,draw =black] (v2) circle (0.3);
		\fill[fill=white,draw =black] (v3) circle (0.3);	\fill[fill=white,draw =black] (v4) circle (0.3);
		\fill[fill=white,draw =black] (v7) circle (0.3);
		\fill[fill=white,draw =black] (v8) circle (0.3);
		\fill[fill=white,draw =black] (v9) circle (0.3);
		\fill[fill=white,draw =black] (v10) circle (0.3);
		\end{tikzpicture}
		\label{fig.2A}}		\qquad
 	\subfloat[ ]{
		\begin{tikzpicture}[scale=0.4]
		\coordinate (v1) at (0,0);
		\coordinate (v2) at (-2,-2);
		\coordinate (v3) at (-2,2);
		\coordinate (v7) at (2,-2);
		\coordinate (v9) at (2,2);	
		\coordinate (v8) at (4,-2);
		\coordinate (v10) at (4,2);
		\coordinate (v4) at (6,0);
		\coordinate (v5) at (8,-2);	
		\coordinate (v6) at (8,2);
		\draw (v1)--(v2);
		\draw (v2)--(v3);
		\draw (v3)--(v1);
		\draw (v1)--(v7);
		\draw[line width=2]  (v7)--(v8);
		\draw[line width=2]  (v8)--(v4);
		\draw [line width=2] (v4)--(v10);	
		\draw[line width=2] (v10)--(v9);
		\draw (v9)--(v1);	
		\fill[fill=white,draw =black] (v1) circle (0.3);
		\fill[fill=white,draw =black] (v2) circle (0.3);
		\fill[fill=white,draw =black] (v3) circle (0.3);	\fill[fill=white,draw =black] (v4) circle (0.3);
		\fill[fill=white,draw =black] (v7) circle (0.3);
		\fill[fill=white,draw =black] (v8) circle (0.3);
		\fill[fill=white,draw =black] (v9) circle (0.3);
		\fill[fill=white,draw =black] (v10) circle (0.3);
		\end{tikzpicture}
		\label{fig.2B}}		\qquad

 \caption{}
	\label{fig.2}
\end{figure}

The graphs in Figure \eqref{fig.3} are obtained by contracting the paths in Figure \eqref{fig.2A} and Figure~\eqref{fig.2B}. The black dots in the figures represent the new vertex of the graph where the path  ``collapses''. 

\begin{figure}[ht]
	\centering
	\subfloat[ ]{
		\begin{tikzpicture}[scale=0.4]
		\coordinate (v1) at (0,0);
		\coordinate (v2) at (-2,-2);
		\coordinate (v3) at (-2,2);
		\coordinate (v7) at (2,-2);
		\coordinate (v9) at (2,2);	
		\coordinate (v8) at (4,-2);
		\coordinate (v10) at (4,2);
		\coordinate (v4) at (6,0);
		\coordinate (v5) at (8,-2);	
		\coordinate (v6) at (8,2);
		\draw (v1)--(v2);
		\draw (v2)--(v3);
		\draw (v3)--(v1);
		\draw (v1)--(v7);
		\draw (v7)--(v8);
		\draw (v8)--(v1);
		\fill[fill=black,draw =black] (v1) circle (0.3);
		\fill[fill=white,draw =black] (v2) circle (0.3);
		\fill[fill=white,draw =black] (v3) circle (0.3);	
		\fill[fill=white,draw =black] (v7) circle (0.3);
		\fill[fill=white,draw =black] (v8) circle (0.3);
		\end{tikzpicture}
		\label{fig.3A}}		\qquad
 	\subfloat[ ]{
		\begin{tikzpicture}[scale=0.4]
		\coordinate (v1) at (0,0);
		\coordinate (v2) at (-2,-2);
		\coordinate (v3) at (-2,2);
		\coordinate (v7) at (2,-2);
		\coordinate (v9) at (2,2);	
		\coordinate (v8) at (4,-2);
		\coordinate (v10) at (4,2);
		\coordinate (v4) at (6,0);
		\coordinate (v5) at (8,-2);	
		\coordinate (v6) at (8,2);
		\draw (v1)--(v2);
		\draw (v2)--(v3);
		\draw (v3)--(v1);
		\draw (v1)--(v7);
		\fill[fill=white,draw =black] (v1) circle (0.3);
		\fill[fill=white,draw =black] (v2) circle (0.3);
		\fill[fill=white,draw =black] (v3) circle (0.3);	
		\fill[fill=black,draw =black] (v7) circle (0.3);
		\end{tikzpicture}
		\label{fig.3B}}
 \caption{ }
	\label{fig.3}
\end{figure}

\end{example}

\begin{remark}\label{rem.simple even and new cycles} 
Note that an even simple path contraction maps even (odd) cycles of $G$ into even (odd) cycles of $G'$ and it does not create new cycles.  
\end{remark}

\begin{lemma}\label{l. codimension of a contraction}
    Let $G=(V,E)$ be a simple graph and $p$ be an even simple path in $G$. Then 
    \[\codim \K[G]=\codim \K[G/p].\]   
\end{lemma}
\begin{proof}
First note that, from Remark \ref{rem.simple even and new cycles}, the graphs $G=(V,E)$ and $G/p=(V',E')$ are both either bipartite or not bipartite. Since $V'=V\setminus p \cup \{y\},$ then, from \cite[Corollary 10.1.21]{V}, in the non bipartite case, we have 
$$\codim \K[G]=|E|-|V|=|E|-|p|-|V|+|p|=|E'|-|V'|=\codim \K[G/p].$$
A similar computation can be made in the bipartite case.
\end{proof}

From the next result we will have that an even simple path contraction do not produce ``new" primitive even closed walks. However, the next lemma does not require the assumption that the path is simple.

\begin{lemma}\label{l.contracting walks}
        Let $G$ be a simple graph. Let $p$ be an even path of $G$ and let $w$ be a primitive even closed walk in $G$.  Then $w/p$ is an even closed walk in $G/p$. On the other hand if $v$ is an even closed walk in $G/p$ then $v=w/p$ for some $w$ an even closed walk in $G$.
\end{lemma}
\begin{proof} Set $p=(x_0,\ldots , x_{2t})$ and recall that we denote by $p^r$ the path $(x_{2t}, \ldots, x_0)$. Let $y\in V'$ be the new vertex and let $\chi:V\to V'$ be the function contracting the path $p$ into $y$. The following cases may be distinguished
 \begin{itemize}
    \item  $w$ and $p$ do not have edges in common. In this case $w/p=\chi(w);$ 
    \item   $w=p|| w'$, where $w'$ is an even walk connecting $x_{2t}$ with $x_{0}$, also $w'$ and $p$ do not have edges in common. Then $w/p=\chi(w')=$ is an even closed walk in $G/p$. 
    \item  $w=p|| w'||p^r||w''$, where $w',w''$ are odd walks both connecting $x_0$ and $x_{2t}$ and they have no edge in common with $p$. Then $w/p=\chi(w')||\chi(w'')$ is an even closed walk in $G/p$.
\end{itemize}
To prove the second part of the statement, let $v$ be a cycle in $G/p$ passing through $y=\chi(p)$, i.e., there are sub-paths of $v$ of the type $(x,y,x')$ where $x,x'\in \mathcal N(x_0) \cup \mathcal N(x_{2t+1}).$
 We explicitly construct a walk $w$ in $G$, by replacing the $y$ in the sequences $(x,y,x')$ contained in $v$ as follows
\begin{itemize}
    \item   if $x,x' \in \mathcal N(x_0)$ then we replace $y$ by $x_0$;   
    \item if $x,x' \in \mathcal N(x_{2t+1})$ then we replace $y$ by $x_{2t}$; 
    \item if  $x \in \mathcal N(x_0)$ and $x' \in \mathcal N(x_{2t})$ then we replace $y$ by $p$;
    \item if  $x' \in \mathcal N(x_0)$ and $x \in \mathcal N(x_{2t})$ then we replace $y$ by $p^r$.
\end{itemize}
     A straightforward calculation proves that the resulting walk $w$ is closed and even in $G$ and~$w/p=v.$  
\end{proof}

We are interested in how a simple even path contraction affects the total Betti numbers of a graph. We first show that  performing this procedure the total Betti  numbers can eventually only go down, then we look for cases where the equality holds. We prove a more general result which does not require that the path is simple.

\begin{theorem}\label{t. main 0}
Let $G$ be a simple graph. Let $G/p$ be a even path contraction of $G$.   Then $$\beta_i(\K[G])\ge \beta_i(\K[G/p])\ \  for\ any\ i \ge 0.$$
\end{theorem}
\begin{proof}
Without loss of generality {(since an even path is a concatenation of paths of length~2)} we assume $p=(e_1,e_2)=(x_1,x_2,x_3)$.  Let $\ell=e_2-e_1\in \K[E]$ be a linear form,  and  let $\mathfrak a=(e_3, \ldots, e_N)\subseteq \K[E]$ be the ideal generated by all the variables except  the two in $p$. We claim that 
$$(\K[G]/\ell \K[G])_{\mathfrak a}\cong \K[G/p],$$ i.e., the edge ring of  $G/p$ is isomorphic to the localization at $\mathfrak a$ of the quotient of $\K[G]$ by $\ell$.
  
The inequality in the statement is a direct consequence of the claim. In fact, the linear form  $\ell$ is regular in $\K[G]$ and the ideal $I_G$ defines a variety with positive dimension in $\mathbb P^{N-1}$. Thus, the operation of taking the quotient  by $\ell$ preserves the total Betti numbers (from the geometrical point of view it is a proper hyperplane section), and the localization can only make them go down.  Indeed, localization is an exact functor and a minimal free resolution remains exact upon localization at $\mathfrak a$. However, after the localization the resolution is not necessarily minimal anymore since some of the maps may have entries not belonging to $\mathfrak a$.  

In order to prove the claim, let $\mathcal G$ be a set of binomial generators, and a universal Gr\"obner basis, for $I_G$. The elements in $\mathcal G$ correspond, from \cite[Proposition 10.1.10]{V},  to the closed primitive even walks of $G.$ Notice that every closed primitive even walk of $G$ which includes either $e_1$ or $e_2$ must include both; moreover, see \cite[Proposition 10.1.8]{V}, it only might include $e_1$ and $e_2$ at most two times.
Thus, there is a natural partition of the set of closed primitive even walks of $G$, that is $W_0\cup W_1\cup W_2$, where the elements 
    in  $W_0$ do not involve $e_1$ and $e_2$;
    the elements in $W_1$  involve only once $e_1$ and $e_2$;
    the elements in $W_2$  involve twice $e_1$ and $e_2$.
The partition of the closed primitive even walks of $G$ is automatically inherited by $\mathcal G$. We set $\mathcal G=\mathcal G_0\cup\mathcal G_1\cup \mathcal G_2$,
where the set $\mathcal G_i$ contains the binomials corresponding to the cycles in $W_i.$

Consider the following isomorphisms  $$ \sfrac{\K[G]}{\ell \K[G]}\cong \sfrac{K[E]}{I_G+(e_2-e_1)}\cong \dfrac{\sfrac{\K[E]}{(e_2-e_1)}}{ \sfrac{I_G+(e_2-e_1)}{(e_2-e_1)}} \cong \sfrac{\K[E\setminus\{e_2\}]}{I'}.$$

The ideal $I'$ is minimally generated by the binomials in $\mathcal G'=\mathcal G_0\cup\mathcal G'_1\cup \mathcal G'_2$
where $\mathcal G_0$ is the same as above, the set $\mathcal G_1'$ contains binomials of the form $e_1f_w$, with $(e_1,e_2)||w\in W_1$; the set $\mathcal G'_2$ contains binomials of the form $e_1^2(f_{w||w'})$ with $(e_1,e_2)||w||(e_2,e_1)||w'\in W_2.$ 

Thus, denoted by $\mathfrak a'=(e_3, \ldots, e_N)\subseteq \K[E\setminus\{e_2\}]$ we have
$$ \left(\sfrac{\K[G]}{\ell \K[G]}\right)_{\mathfrak a}\cong \left(\sfrac{\K[E\setminus\{e_2\}]}{I'}\right)_{\mathfrak a'}\cong \sfrac{\K(e_1)[E\setminus\{e_1, e_2\}]}{I''}.$$
The field $\K(e_1)$ is the extension over $\K$ generated by $e_1$.
Localizing, the form $e_1$ becomes a unit, so the ideal  $I''$ in $\K(e_1)[E\setminus\{e_1,e_2\}]$ is (not necessarily minimally) generated by the binomials in $\mathcal G''=\mathcal G_0\cup\mathcal G''_1\cup \mathcal G''_2$
where $\mathcal G_0$ is again the same as above and the set $\mathcal G_1'$ contains binomials of the form $f_w$, with $(e_1,e_2)||w\in W_1$; the set $\mathcal G'_2$ contains binomials of the form $(f_{w||w'})$ with $(e_1,e_2)||w||(e_2,e_1)||w'\in W_2.$ Since $e_1$ does not appears in the elements of the set $\mathcal G''$,  the same binomials define an ideal in $[E\setminus\{e_1, e_2\}]$.    
Such ideal, from the description of $\mathcal G''$ and by Lemma \ref{l.contracting walks}, is precisely $I_{G/p}$.\end{proof}

\begin{corollary}\label{c. not CM contraction} Let $G$ be a simple graph. Let $p$ be a even simple path of $G$. 
    If $\K[G/p]$ is not Cohen-Macaulay then $\K[G]$  is not Cohen-Macaulay.
\end{corollary}
\begin{proof}
It follows from Lemma \ref{l. codimension of a contraction} and Theorem \ref{t. main 0}.
\end{proof}
We will show in Examples \ref{e. contr 3->1} and \ref{e. contr 2->0} that the converse of Corollary \ref{c. not CM contraction} is false, i.e., for an even simple path $p$ of $G$, having $\K[G]$  not Cohen-Macaulay does not necessarily imply that $\K[G/p]$ is Cohen-Macaulay.

It would be interesting to find extra hypotheses in Theorem \ref{t. main 0} to guarantee the equality of the total Betti number. In this direction we have the following result. A combinatorial proof of it is given in Section~\ref{s. proof of main1}.

\begin{theorem}\label{t. main 1}
Let $G$ be a simple graph. Let $G/p$ be a even simple path contraction of $G$.  If  $p\subset q$ where $q$ is a path in $G$ such that $|q|\ge |p|+2$,  then
$$\beta_i(\K[G])= \beta_i(\K[G/p])\ \ for\ any\ i \ge 0.$$
\end{theorem}
The next corollary collects immediate consequences of Theorem \ref{t. main 1}.
\begin{corollary}
Let $G$ be a simple graph. Let $p$ be a even simple path of $G$ such that $p\subset q$ where $q$ is a path in $G$ such that $|q|\ge |p|+2$. Then
\begin{itemize}
    \item[(i)] $\K[G]$ is Cohen-Macaulay if and only if $\K[G/p]$ is Cohen-Macaulay;
    \item[(ii)] $\K[G]$ is Gorenstein if and only if $\K[G/p]$ is Gorenstein;
\end{itemize}    
\end{corollary}

\begin{remark} The graph $G$ in Figure \eqref{fig.contr1A} consists of a path $q$ of length at least $4$ and a dashed part that could be any simple graph.
The path $p$ consists of any two consecutive vertices of $q$. The graph in Figure \eqref{fig.contr1B} is $G/p$ obtained by  contracting  $p$. 
From Theorem \ref{t. main 1} the two graphs in Figure \eqref{fig.contr1} have the same total Betti numbers.

		\begin{figure}[ht]
	\centering
	\subfloat[ ]{
		\begin{tikzpicture}[scale=0.35]
		\coordinate (v0) at (-2,1.5);
		\coordinate (v1) at (0,0-1);
		\coordinate (v2) at (0,4);
		\coordinate (v3) at (3,-1);
		\coordinate (v4) at (3,4);	
		\coordinate (v5) at (6,5);
		\coordinate (v6) at (6,3);
		\coordinate (v7) at (6,1);
		\coordinate (v8) at (6,-1);
		\coordinate (v9) at (6,-3);	
		\draw[line width=2] (v1)--(v0);
		\draw (v0)--(v2);
		\draw[line width=2] (v1)--(v3);
		\draw (v2)--(v4);
		\draw[dotted] (v4)--(v5);
		\draw[dotted] (v4)--(v6);
		\draw[dotted] (v4)--(v7);
		\draw[dotted] (v3)--(v7);
		\draw[dotted] (v3)--(v8);
		\draw[dotted] (v3)--(v9);
		
		\fill[fill=white,draw =black] (v0) circle (0.3);
		\fill[fill=white,draw =black] (v1) circle (0.3);
		\fill[fill=white,draw =black] (v2) circle (0.3);
		\fill[fill=white,draw =black] (v3) circle (0.3);
		\fill[fill=white,draw =black] (v4) circle (0.3);
		\end{tikzpicture}
		\label{fig.contr1A}}		\qquad
	\subfloat[ ]{
		\begin{tikzpicture}[scale=0.35]
\coordinate (v0) at (0,1.5);
\coordinate (v1) at (0,0-1);
\coordinate (v2) at (0,4);
\coordinate (v3) at (3,-1);
\coordinate (v4) at (3,4);	
\coordinate (v5) at (6,5);
\coordinate (v6) at (6,3);
\coordinate (v7) at (6,1);
\coordinate (v8) at (6,-1);
\coordinate (v9) at (6,-3);	
\draw (v0)--(v3);
\draw (v0)--(v4);

\draw[dotted] (v4)--(v5);
\draw[dotted] (v4)--(v6);
\draw[dotted] (v4)--(v7);
\draw[dotted] (v3)--(v7);
\draw[dotted] (v3)--(v8);
\draw[dotted] (v3)--(v9);

\fill[fill=white,draw =black] (v0) circle (0.3);
\fill[fill=black,draw =black] (v3) circle (0.3);
\fill[fill=white,draw =black] (v4) circle (0.3);
		\end{tikzpicture}
		
		\label{fig.contr1B}}
	\caption{Contraction of  a path of length two contained in a simple path of length four.}
	\label{fig.contr1}
\end{figure}
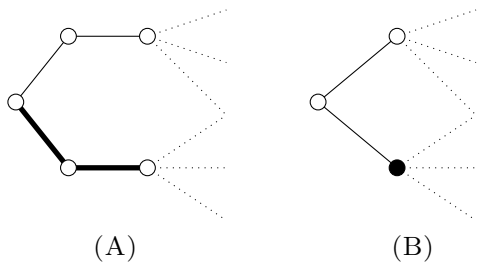

\end{remark}
Does an analogous of Theorem \ref{t. main 1} exist with odd contractions? In the next example we show that the answer to such question is in general negative. Indeed, we give an example of an odd simple path contraction which leads to non comparable Betti numbers.

\begin{example}\label{e. contr 3->2}
Consider the graphs $G$ and $G'$ in Figure \eqref{fig.oddcontr1}. Note that $G'$ has a path $q$ of length $3$ and $G'$ is obtained by  $G$ by contracting an edge in $q$.

		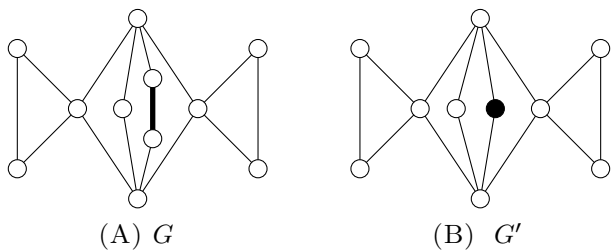
\begin{figure}[ht]
	\centering
	\subfloat[$G$]{
		\begin{tikzpicture}[scale=0.4]
		\coordinate (v1) at (0,0);
		\coordinate (v2) at (0,4);
		\coordinate (v3) at (2,2);
		\coordinate (v4) at (4,5);
		\coordinate (v5) at (4,-1);	
		\coordinate (v6) at (6,2);
		\coordinate (v7) at (8,0);
		\coordinate (v8) at (8,4);
		\coordinate (v9) at (3.5,2);
		\coordinate (v10) at (4.5,3);	
		\coordinate (v11) at (4.5,1);
		\draw (v1)--(v2);
		\draw (v3)--(v2);
		\draw (v1)--(v3);
		\draw (v3)--(v4);
		\draw (v3)--(v5);
		\draw (v5)--(v6);
		\draw (v4)--(v6);
		\draw (v7)--(v6);	
		\draw (v8)--(v6);
		\draw (v7)--(v8);	
		
		\draw (v4)--(v9);
		\draw (v9)--(v5);
		
		\draw (v4)--(v10);
		\draw[line width=2] (v11)--(v10);
		\draw (v5)--(v11);
		\fill[fill=white,draw =black] (v1) circle (0.3);
		\fill[fill=white,draw =black] (v2) circle (0.3);
		\fill[fill=white,draw =black] (v3) circle (0.3);	\fill[fill=white,draw =black] (v4) circle (0.3);
		\fill[fill=white,draw =black] (v5) circle (0.3);
		\fill[fill=white,draw =black] (v6) circle (0.3);
		\fill[fill=white,draw =black] (v7) circle (0.3);
		\fill[fill=white,draw =black] (v8) circle (0.3);
		\fill[fill=white,draw =black] (v9) circle (0.3);
		\fill[fill=white,draw =black] (v10) circle (0.3);
		\fill[fill=white,draw =black] (v11) circle (0.3);
		\end{tikzpicture}
		\label{fig.oddcontr1A}}		\qquad
	\subfloat[ $G'$]{
		\begin{tikzpicture}[scale=0.4]
		\coordinate (v1) at (0,0);
		\coordinate (v2) at (0,4);
		\coordinate (v3) at (2,2);
		\coordinate (v4) at (4,5);
		\coordinate (v5) at (4,-1);	
		\coordinate (v6) at (6,2);
		\coordinate (v7) at (8,0);
		\coordinate (v8) at (8,4);
		\coordinate (v9) at (3.2,2);
		\coordinate (v10) at (4.5,2);	
		\draw (v1)--(v2);
		\draw (v3)--(v2);
		\draw (v1)--(v3);
		\draw (v3)--(v4);
		\draw (v3)--(v5);
		\draw (v5)--(v6);
		\draw (v4)--(v6);
		\draw (v7)--(v6);	
		\draw (v8)--(v6);
		\draw (v7)--(v8);	
		
		\draw (v4)--(v9);
		\draw (v9)--(v5);
		
		\draw (v4)--(v10);
		\draw (v5)--(v10);
		\fill[fill=white,draw =black] (v1) circle (0.3);
		\fill[fill=white,draw =black] (v2) circle (0.3);
		\fill[fill=white,draw =black] (v3) circle (0.3);	\fill[fill=white,draw =black] (v4) circle (0.3);
		\fill[fill=white,draw =black] (v5) circle (0.3);
		\fill[fill=white,draw =black] (v6) circle (0.3);
		\fill[fill=white,draw =black] (v7) circle (0.3);
		\fill[fill=white,draw =black] (v8) circle (0.3);
		\fill[fill=white,draw =black] (v9) circle (0.3);
		\fill[fill=black,draw =black] (v10) circle (0.3);
		\end{tikzpicture}
		
		\label{fig.oddcontr1B}}
	\caption{The graph $G'$ in Example \ref{e. contr 3->2} is an edge contraction of $G$.}
	\label{fig.oddcontr1}
\end{figure}

A computation with CoCoA shows that the total Betti numbers are not comparable,
$$\begin{array}{r|ccccc}
    i: &  0 &  1 &  2 &  3 &  4\\
\hline
\beta_i(\K[G]) & 1   & 8 & 18 & 16 & 5 \\[10pt]
\hline
\beta_i(\K[G']) & 1   & 9 & 19 & 9 & 1 \\[10pt]
\end{array}.
$$
\end{example}

It is also natural to ask if an analogous of Theorem \ref{t. main 1} holds with an even contraction and a weaker assumption on $q$. For instance if either $|q|=|p|+1$ (see Figure \eqref{fig.contr2}) or $p$ is a maximal path and $q=p$ (see Figure \eqref{fig.contr3}).   In the next examples we show that in these cases the equality does not necessarily hold anymore.  These examples also show that an even contraction may produce a Cohen-Macaulay graph from a non Cohen-Macaulay one.

		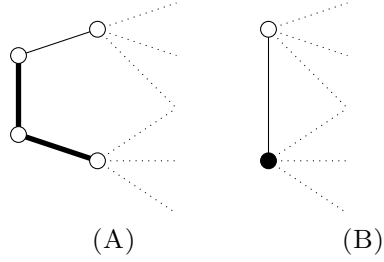
\begin{figure}[ht]
	\centering
	\subfloat[ ]{
			\begin{tikzpicture}[scale=0.35]
		\coordinate (v1) at (0,0);
		\coordinate (v2) at (0,3);
		\coordinate (v3) at (3,-1);
		\coordinate (v4) at (3,4);	
		\coordinate (v5) at (6,5);
		\coordinate (v6) at (6,3);
		\coordinate (v7) at (6,1);
		\coordinate (v8) at (6,-1);
		\coordinate (v9) at (6,-3);	
		\draw[line width=2] (v1)--(v2);
		\draw[line width=2] (v1)--(v3);
		\draw (v2)--(v4);
		\draw[dotted] (v4)--(v5);
		\draw[dotted] (v4)--(v6);
		\draw[dotted] (v4)--(v7);
		\draw[dotted] (v3)--(v7);
		\draw[dotted] (v3)--(v8);
		\draw[dotted] (v3)--(v9);
		
		\fill[fill=white,draw =black] (v1) circle (0.3);
		\fill[fill=white,draw =black] (v2) circle (0.3);
		\fill[fill=white,draw =black] (v3) circle (0.3);
		\fill[fill=white,draw =black] (v4) circle (0.3);
		\end{tikzpicture}
					\label{fig.contr2A}}		\qquad
	\subfloat[ ]{
			\begin{tikzpicture}[scale=0.35]
		\coordinate (v3) at (3,-1);
		\coordinate (v4) at (3,4);	
		\coordinate (v5) at (6,5);
		\coordinate (v6) at (6,3);
		\coordinate (v7) at (6,1);
		\coordinate (v8) at (6,-1);
		\coordinate (v9) at (6,-3);	
		\draw (v3)--(v4);
		\draw[dotted] (v4)--(v5);
		\draw[dotted] (v4)--(v6);
		\draw[dotted] (v4)--(v7);
		\draw[dotted] (v3)--(v7);
		\draw[dotted] (v3)--(v8);
		\draw[dotted] (v3)--(v9);
		
		\fill[fill=black,draw =black] (v3) circle (0.3);
		\fill[fill=white,draw =black] (v4) circle (0.3);
		\end{tikzpicture}
		
		\label{fig.contr2B}}
	\caption{Contraction of  a path of length two in a simple path of length~three.}
	\label{fig.contr2}
\end{figure}

		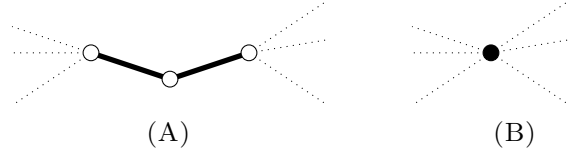
\begin{figure}[ht]
	\centering
	\subfloat[ ]{
		\begin{tikzpicture}[scale=0.35]
		\coordinate (v1) at (0,0);
		\coordinate (v2) at (3,1);
		\coordinate (v3) at (-3,1);
		\coordinate (v4) at (6,3);
		\coordinate (v5) at (6,1.5);
		\coordinate (v6) at (6,-1);
		\coordinate (v7) at (-6,2);
		\coordinate (v8) at (-6,1);	
		\coordinate (v9) at (-6,-1);			
		\draw[line width=2] (v1)--(v2);
		\draw[line width=2] (v1)--(v3);
		\draw[dotted] (v2)--(v4);
		\draw[dotted] (v2)--(v5);
		\draw[dotted] (v2)--(v6);
		\draw[dotted] (v3)--(v7);
		\draw[dotted] (v3)--(v8);
		\draw[dotted] (v3)--(v9);
		
		\fill[fill=white,draw =black] (v1) circle (0.3);
		\fill[fill=white,draw =black] (v2) circle (0.3);
		\fill[fill=white,draw =black] (v3) circle (0.3);
		\end{tikzpicture}
		\label{fig.contr3A}}		\qquad
	\subfloat[ ]{
\begin{tikzpicture}[scale=0.35]
\coordinate (v1) at (0,1);
\coordinate (v4) at (3,3);
\coordinate (v5) at (3,1.5);
\coordinate (v6) at (3,-1);
\coordinate (v7) at (-3,2);
\coordinate (v8) at (-3,1);	
\coordinate (v9) at (-3,-1);			
\draw[dotted] (v1)--(v4);
\draw[dotted] (v1)--(v5);
\draw[dotted] (v1)--(v6);
\draw[dotted] (v1)--(v7);
\draw[dotted] (v1)--(v8);
\draw[dotted] (v1)--(v9);

\fill[fill=black,draw =black] (v1) circle (0.3);
\end{tikzpicture}		
		\label{fig.contr3B}}
	\caption{Contraction of a simple path of length two}
	\label{fig.contr3}
\end{figure}

\begin{example}\label{e. contr 3->1}
Let $G=(V,E)$ be the graph with vertex set $V=\{x_1, \ldots ,x_{10} \}$ and edges
\[\begin{array}{rl}
    E= &  \{ \{x_1,x_2 \},\{x_2,x_3 \}, \{x_1,x_3 \}, \{x_4,x_5 \},\{x_5,x_6 \}, \{x_4,x_6 \}   \}\cup\\
     &  \{ \{x_1,x_7 \},\{x_7,x_8 \}, \{x_8,x_4 \}, \{x_1,x_9 \},\{x_9,x_{10} \}, \{x_{10},x_4 \}\}.
\end{array}
 \]

Set $q=(x_1, x_9, x_{10}, x_4 )$ and $p=(x_1, x_9, x_{10})$, then $p$ is a simple even path contained in a path of legth $|p|+1$, but one can check that the Betti numbers strictly decrease after the contraction. Indeed, $\K[G]$ is not Cohen-Macaulay but $\K[G/p]$ is Cohen-Macaulay.
\[
\begin{array}{l|cccc}
i & 0 & 1 & 2 & 3   \\
\hline
    \beta_i (\K[G]) & 1 & 4 & 4 & 1  \\[7pt]
\hline
    \beta_i (\K[G/p]) & 1 & 2 & 1 & 
\end{array}
\]

		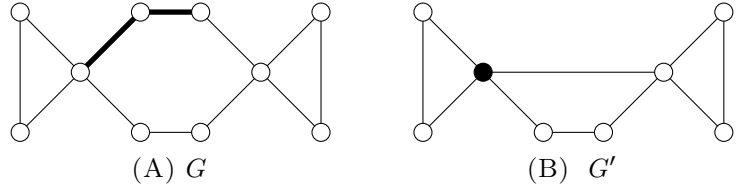
\begin{figure}[ht]
	\centering
	\subfloat[$G$]{
		\begin{tikzpicture}[scale=0.4]
		\coordinate (v1) at (0,0);
		\coordinate (v2) at (-2,-2);
		\coordinate (v3) at (-2,2);
		\coordinate (v7) at (2,-2);
		\coordinate (v9) at (2,2);	
		\coordinate (v8) at (4,-2);
		\coordinate (v10) at (4,2);
		\coordinate (v4) at (6,0);
		\coordinate (v5) at (8,-2);	
		\coordinate (v6) at (8,2);
		\draw (v1)--(v2);
		\draw (v2)--(v3);
		\draw (v3)--(v1);
		\draw (v1)--(v7);
		\draw (v7)--(v8);
		\draw (v8)--(v4);
		\draw (v4)--(v10);	
		\draw[line width=2] (v10)--(v9);
		\draw[line width=2] (v9)--(v1);	
		\draw (v4)--(v5);
		\draw (v5)--(v6);	
		\draw (v6)--(v4);
		\fill[fill=white,draw =black] (v1) circle (0.3);
		\fill[fill=white,draw =black] (v2) circle (0.3);
		\fill[fill=white,draw =black] (v3) circle (0.3);	\fill[fill=white,draw =black] (v4) circle (0.3);
		\fill[fill=white,draw =black] (v5) circle (0.3);
		\fill[fill=white,draw =black] (v6) circle (0.3);
		\fill[fill=white,draw =black] (v7) circle (0.3);
		\fill[fill=white,draw =black] (v8) circle (0.3);
		\fill[fill=white,draw =black] (v9) circle (0.3);
		\fill[fill=white,draw =black] (v10) circle (0.3);
		\end{tikzpicture}
		\label{fig.CMcontr1A}}		\qquad
	\subfloat[ $G'$]{
		\begin{tikzpicture}[scale=0.4]
		\coordinate (v1) at (0,0);
		\coordinate (v2) at (-2,-2);
		\coordinate (v3) at (-2,2);
		\coordinate (v7) at (2,-2);
		\coordinate (v8) at (4,-2);
		\coordinate (v10) at (3,2);
		\coordinate (v4) at (6,0);
		\coordinate (v5) at (8,-2);	
		\coordinate (v6) at (8,2);
		\draw (v1)--(v2);
		\draw (v2)--(v3);
		\draw (v3)--(v1);
		\draw (v1)--(v7);
		\draw (v7)--(v8);
		\draw (v8)--(v4);
		\draw (v4)--(v1);	
		\draw (v4)--(v5);
		\draw (v5)--(v6);	
		\draw (v6)--(v4);
		\fill[fill=black,draw =black] (v1) circle (0.3);
		\fill[fill=white,draw =black] (v2) circle (0.3);
		\fill[fill=white,draw =black] (v3) circle (0.3);	\fill[fill=white,draw =black] (v4) circle (0.3);
		\fill[fill=white,draw =black] (v5) circle (0.3);
		\fill[fill=white,draw =black] (v6) circle (0.3);
		\fill[fill=white,draw =black] (v7) circle (0.3);
		\fill[fill=white,draw =black] (v8) circle (0.3);
		\end{tikzpicture}
		
		\label{fig.CMcontr1B}}
	\caption{The graphs in Example \ref{e. contr 3->1}.}
	\label{fig.CMcontr1}
\end{figure}

\end{example}

With a slight modification of the graph in Example \ref{e. contr 3->1} we show that contracting a maximal path may the total Betti numbers go down.
\begin{example}\label{e. contr 2->0}
Let $G'=(V',E')$ be the graph with vertex set $V=\{x_1, \ldots ,x_{9} \}$ and edges
\[\begin{array}{rl}
    E= &  \{ \{x_1,x_2 \},\{x_2,x_3 \}, \{x_1,x_3 \}, \{x_4,x_5 \},\{x_5,x_6 \}, \{x_4,x_6 \}   \}\cup\\
     &  \{ \{x_1,x_7 \},\{x_7,x_8 \}, \{x_8,x_4 \}, \{x_1,x_9 \},\{x_9,x_{4} \}.
\end{array}
 \]
Note that $G'$ is isomorphic to the graph obtained from the graph $G$ in Example \ref{e. contr 3->1} by contracting the edge $\{x_9,x_{10}\}$, $G'\cong G/\{x_9,x_{10}\}$.

Set $p'=(x_1, x_9, x_4 )$, thus $p'$ is a maximal simple even path. Again one can check that the Betti numbers strictly decrease after the contraction. Indeed $\K[G']$ is not Cohen-Macaulay but $\K[G'/p']$ has this property. 
\[
\begin{array}{l|cccc}
i & 0 & 1 & 2 & 3   \\
\hline
    \beta_i (\K[G']) & 1 & 4 & 4 & 1  \\[7pt]
\hline
    \beta_i (\K[G'/p']) & 1 & 3 & 2 & 
\end{array}
\]
		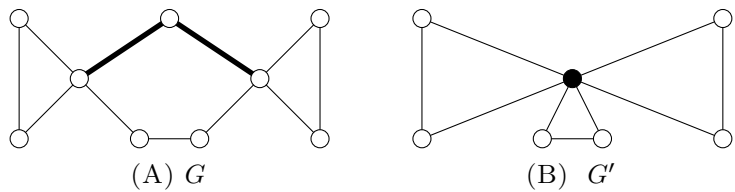
\begin{figure}[ht]
	\centering
	\subfloat[$G$]{
		\begin{tikzpicture}[scale=0.4]
		\coordinate (v1) at (0,0);
		\coordinate (v2) at (-2,-2);
		\coordinate (v3) at (-2,2);
		\coordinate (v7) at (2,-2);
		\coordinate (v9) at (3,2);	
		\coordinate (v8) at (4,-2);
		\coordinate (v10) at (4,2);
		\coordinate (v4) at (6,0);
		\coordinate (v5) at (8,-2);	
		\coordinate (v6) at (8,2);
		\draw (v1)--(v2);
		\draw (v2)--(v3);
		\draw (v3)--(v1);
		\draw (v1)--(v7);
		\draw (v7)--(v8);
		\draw (v8)--(v4);
		\draw[line width=2] (v4)--(v9);
		\draw[line width=2] (v9)--(v1);	
		\draw (v4)--(v5);
		\draw (v5)--(v6);	
		\draw (v6)--(v4);
		\fill[fill=white,draw =black] (v1) circle (0.3);
		\fill[fill=white,draw =black] (v2) circle (0.3);
		\fill[fill=white,draw =black] (v3) circle (0.3);	\fill[fill=white,draw =black] (v4) circle (0.3);
		\fill[fill=white,draw =black] (v5) circle (0.3);
		\fill[fill=white,draw =black] (v6) circle (0.3);
		\fill[fill=white,draw =black] (v7) circle (0.3);
		\fill[fill=white,draw =black] (v8) circle (0.3);
		\fill[fill=white,draw =black] (v9) circle (0.3);
		\end{tikzpicture}
		\label{fig.CM2contr1A}}		\qquad
	\subfloat[ $G'$]{
		\begin{tikzpicture}[scale=0.4]
		\coordinate (v1) at (3,0);
		\coordinate (v2) at (-2,-2);
		\coordinate (v3) at (-2,2);
		\coordinate (v7) at (2,-2);
		\coordinate (v8) at (4,-2);
		\coordinate (v5) at (8,-2);	
		\coordinate (v6) at (8,2);
		\draw (v1)--(v2);
		\draw (v2)--(v3);
		\draw (v3)--(v1);
		\draw (v1)--(v7);
		\draw (v7)--(v8);
		\draw (v8)--(v1);
		\draw (v1)--(v5);
		\draw (v5)--(v6);	
		\draw (v6)--(v1);
		\fill[fill=black,draw =black] (v1) circle (0.3);
		\fill[fill=white,draw =black] (v2) circle (0.3);
		\fill[fill=white,draw =black] (v3) circle (0.3);	
		\fill[fill=white,draw =black] (v5) circle (0.3);
		\fill[fill=white,draw =black] (v6) circle (0.3);
		\fill[fill=white,draw =black] (v7) circle (0.3);
		\fill[fill=white,draw =black] (v8) circle (0.3);
		\end{tikzpicture}
		
		\label{fig.CM2contr1B}}
	\caption{The graphs in Example \ref{e. contr 2->0}.}
	\label{fig.CM2contr1}
\end{figure}

\end{example}

Theorem \ref{t. main 0} does not hold if we consider an even walk instead of even path. In fact, a contraction of a walk in a graph could produce new primitive odd cycles  which make the total Betti numbers increase. This phenomenon is illustrated in the next example.
\begin{example}\label{ex.betti increase}
Let $G=(V,E)$ be the graph with vertex set $V=\{x_1,\ldots, x_{11}\}$ and edge set
\[
\begin{array}{rl}
     E=& \{\{ x_1,x_2\},\{ x_2,x_3\},\{ x_3,x_4\},\{ x_4,x_5\},\{ x_5,x_6\},\{ x_1,x_6\},\{ x_2,x_5\}\}\cup \\
     & \{\{ x_1,x_7\},\{ x_{1},x_{8}\},\{ x_7,x_9\},\{ x_8,x_{9}\},\{ x_{9},x_{10}\},\{ x_9,x_{11}\},\{ x_{10},x_{11}\}\}
\end{array}
\]
The contraction of the walk $p=(x_1,x_2,x_3)$ gives the graph $G'=G/p$
whose edge set is
\[
\begin{array}{rl}
     E'=& \{\{ x_1,x_4\},\{ x_4,x_5\},\{ x_5,x_6\},\{ x_1,x_6\},\{ x_2,x_5\}\}\cup \\
     & \{\{ x_1,x_7\},\{ x_{1},x_{8}\},\{ x_7,x_9\},\{ x_8,x_{9}\},\{ x_{9},x_{10}\},\{ x_9,x_{11}\},\{ x_{10},x_{11}\}\}
\end{array}
\]
These graphs are pictured in Figure \eqref{fig.walkcontr1}. 
A computation shows that the total Betti numbers of $G'$ are bigger than those of $G$, precisely we get  
$$\begin{array}{r|ccccc}
    i: &  0 &  1 &  2 &  3 &  4\\
\hline
\beta_i(\K[G]) & 1   & 3 & 3 & 1 &  \\[10pt]
\hline
\beta_i(\K[G']) & 1   & 8 & 15 & 10 & 2 \\[10pt]
\end{array}.
$$
		\begin{figure}[ht]
	\centering
	\subfloat[The graph  $G$ in Example \ref{ex.betti increase}]{
		\begin{tikzpicture}[scale=0.4]
		\coordinate (v1) at (0,0);
		\coordinate (v2) at (4,0);
		\coordinate (v3) at (8,0);
		\coordinate (v4) at (8,4);
		\coordinate (v5) at (4,4);	
		\coordinate (v6) at (0,4);
		\coordinate (v7) at (12,0);
		\coordinate (v8) at (12,4);
		\coordinate (v9) at (16,0);
		\coordinate (v10) at (20,0);	
		\coordinate (v11) at (20,4);
		\draw[line width=2] (v1)--(v2);
		\draw[line width=2] (v2)--(v3);
		\draw (v3)--(v4);
		\draw (v3)--(v4);
		\draw (v4)--(v5);
		\draw (v5)--(v6);
		\draw (v6)--(v1);
		\draw (v2)--(v5);	
		\draw (v3)--(v7);
		\draw (v3)--(v8);	
		\draw (v7)--(v9);
		\draw (v8)--(v9);	
		\draw (v9)--(v10);
		\draw (v9)--(v11);
		\draw (v11)--(v10);
		\fill[fill=white,draw =black] (v1) circle (0.3);
		\fill[fill=white,draw =black] (v2) circle (0.3);
		\fill[fill=white,draw =black] (v3) circle (0.3);	\fill[fill=white,draw =black] (v4) circle (0.3);
		\fill[fill=white,draw =black] (v5) circle (0.3);
		\fill[fill=white,draw =black] (v6) circle (0.3);
		\fill[fill=white,draw =black] (v7) circle (0.3);
		\fill[fill=white,draw =black] (v8) circle (0.3);
		\fill[fill=white,draw =black] (v9) circle (0.3);
		\fill[fill=white,draw =black] (v10) circle (0.3);
		\fill[fill=white,draw =black] (v11) circle (0.3);
		\end{tikzpicture}
		\label{fig.walkcontr1A}}		\qquad
	\subfloat[The graph  $G'$ in Example \ref{ex.betti increase}]{
		\begin{tikzpicture}[scale=0.4]
		\coordinate (v3) at (8,0);
		\coordinate (v4) at (8,4);
		\coordinate (v5) at (4,4);	
		\coordinate (v6) at (0,4);
		\coordinate (v7) at (12,0);
		\coordinate (v8) at (12,4);
		\coordinate (v9) at (16,0);
		\coordinate (v10) at (20,0);	
		\coordinate (v11) at (20,4);

		\draw (v3)--(v4);
		\draw (v3)--(v4);
		\draw (v4)--(v5);
		\draw (v5)--(v6);
		\draw (v6)--(v3);
		\draw (v3)--(v5);	
		\draw (v3)--(v7);
		\draw (v3)--(v8);	
		\draw (v7)--(v9);
		\draw (v8)--(v9);	
		\draw (v9)--(v10);
		\draw (v9)--(v11);
		\draw (v11)--(v10);		
		\fill[fill=black,draw =black] (v3) circle (0.3);	\fill[fill=white,draw =black] (v4) circle (0.3);
		\fill[fill=white,draw =black] (v5) circle (0.3);
		\fill[fill=white,draw =black] (v6) circle (0.3);
		\fill[fill=white,draw =black] (v7) circle (0.3);
		\fill[fill=white,draw =black] (v8) circle (0.3);
		\fill[fill=white,draw =black] (v9) circle (0.3);
		\fill[fill=white,draw =black] (v10) circle (0.3);
		\end{tikzpicture}
		
		\label{fig.walkcontr1B}}
	\caption{$G'$ is a walk contraction of $G$.}
	\label{fig.walkcontr1}
\end{figure}
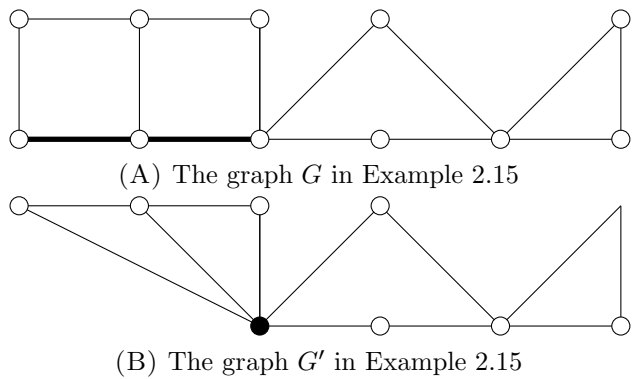
\end{example}


\section{Proof of Theorem \ref{t. main 1}}\label{s. proof of main1}

In this section we prove Theorem \ref{t. main 1} by using combinatorial tools. 
We begin by introducing some terminology and results about simplicial complexes. For what is not included in this section we refer to  \cite{HH-monomial}, see in particular \cite [Section 1.5.1]{HH-monomial} for the background and \cite[Section 5.1.4]{HH-monomial} for the results on simplicial homology. 
Then, we  prove some preliminary results in order to apply the formula in {\cite[Theorem 2.1]{BCMP}} which allows us to compute the total Betti numbers of the toric ideal of a graph.   

\subsection{Simplicial Complexes}\hfill\\
A {\it simplicial complex} $\Delta$ on $[n]$ is a finite collection of  subsets of $[n]$ with the property that $A\in \Delta$ implies $B\in \Delta$ for any $B\subseteq A$. We will assume $n$ large enough and we  do not require that $\{i\}\in \Delta$ for all $i\in [n]$.

Let $\Delta$ be a simplicial complex on $[n]$ and let $A\subseteq [n]$ be a non empty set, we call the {\it cone} of $\Delta$ over $A$, and we denote by $A*\Delta$, the simplicial complex whose facets are
$A\cup F$, where $F$ is a facet of $\Delta.$ It is well known that cones are acyclic, i.e., their reduced simplicial homology groups are trivial in all dimensions, see for instance \cite{stanley}.

We need the following technical Lemma. It is an immediate consequence of the well known reduced Mayer–Vietoris exact sequence, see \cite[Proposition 5.1.8]{HH-monomial}. We include a short proof for the convenience of the reader.

\begin{lemma}\label{l.m-v} Let $\Delta_1 , \Delta_2$ be two acyclic simplicial complexes on $[n]$. Set $\Delta= \Delta_1 \cup \Delta_2$.  Then
	$$\tilde H_i(\Delta;\K)\cong\tilde H_{i-1}(\Delta_1\cap \Delta_2;\K)\ \ \ \ \ \forall i>0.$$
\end{lemma}
\begin{proof}
	The statement follows from the long exact sequence,
	$$  \to \tilde H_i(\Delta_1;\K)\oplus \tilde H_i(\Delta_2;\K) \to \tilde H_i(\Delta;\K)\to \tilde H_{i-1}(\Delta_1\cap \Delta_2;\K)\to \tilde H_{i-1}(\Delta_1;\K)\oplus \tilde H_{i-1}(\Delta_2;\K) \to  $$
	since, by hypothesis, $\tilde H_i(\Delta_1;\K)= \tilde H_i(\Delta_2;\K)=H_{i-1}(\Delta_1;\K)= \tilde H_{i-1}(\Delta_2;\K)=0.$ 
\end{proof}

The following results are corollaries of Lemma \ref{l.m-v}.

\begin{corollary}\label{c. m-v case 1}
	Let $\Delta_1, \Delta_2$ be simplicial complex on $[n]$. Let $A,B,A',B'\subseteq [n]$ be non empty sets such that 
 $$A\cap B=A'\cap B'=(A\cup B)\cap (\Delta_1 \cup \Delta_2)=(A'\cup B')\cap (\Delta_1 \cup \Delta_2)=\emptyset.$$
Then 
   $$\tilde H_i\big((A*\Delta_1) \cup (B*\Delta_2);\K\big)\cong \tilde H_i\big((A'*\Delta_1) \cup (B'*\Delta_2);\K\big)\ \ \ \ \ \forall i>0.$$
\end{corollary}
\begin{proof}
	Since we have $$(A*\Delta_1) \cap (B*\Delta_2)=(A'*\Delta_1) \cap (B'*\Delta_2)=\Delta_1 \cap \Delta_2,$$
	it follows from Lemma \ref{l.m-v}
and from the well known fact that cones are acyclic.
\end{proof}

In the next corollary we assume $B=B'$ and we make the other assumptions weaker.

\begin{corollary}\label{c. m-v case 2}
	Let $\Delta_1, \Delta_2$ be simplicial complexes on $[n].$ Let $A,A',B$  be non empty subsets of $[n]$ such that 
	$$A\cap B=A'\cap B= A\cap (\Delta_1 \cup \Delta_2)=A'\cap (\Delta_1 \cup \Delta_2)=\emptyset.$$

	Then $$\tilde H_i\left((A*\Delta_1) \cup (A*\Delta_2);\K\right)\cong \tilde H_i\left((A'*\Delta_1) \cup (B*\Delta_2);\K\right)\ \ \ \ \ \forall i>0.$$
\end{corollary}
\begin{proof}
	It follows from Lemma \ref{l.m-v} since $$(A*\Delta_1) \cap (B*\Delta_2)=(A'*\Delta_1) \cap (B*\Delta_2)=(\Delta_1 \cap \Delta_2) \cup (B \cap \Delta_1).$$
\end{proof}

\begin{proposition}\label{p. m-v case 3}
	Let, $\Delta_1,\Delta_2, \Delta_3$ be simplicial complexes on $[n]$.  For any $A,B$ non empty subsets of $[n]$, consider the following simplicial complex
	$$\Delta_{A,B} = A*\Delta_1 \cup B*\Delta_2\cup (A\cup B)*\Delta_3.$$
	If $A,B,A',B'\subseteq [n]$ are non empty and  such that
 	$$ 	A\cap B=(A\cup B)\cap (\Delta_1 \cup \Delta_2\cup \Delta_3)=A'\cap B'=(A'\cup B')\cap (\Delta_1 \cup \Delta_2\cup \Delta_3)=\emptyset,$$
then	
	 $$\tilde H_i(\Delta_{A,B};\K)\cong \tilde H_i(\Delta_{A',B'};\K)\ \ \ \ \ \forall i>0.$$
\end{proposition}
\begin{proof}
	Let $A,B\subseteq [n]$ be finite non empty, disjoint and such that $(A\cup B)\cap (\Delta_1 \cup \Delta_2\cup \Delta_3)=\emptyset$. 
	Note that $\Delta_{A,B}$ can be written as follow
	$$\Delta_{A,B} =(A*\Delta_1) \cup B*(\Delta_2\cup A*\Delta_3).$$
	So, by Lemma \ref{l.m-v},  we get
	\begin{equation}\label{eq.H1}
	\tilde H_i(\Delta_{A,B};\K)\cong\tilde H_{i-1}\left((\Delta_1\cap \Delta_2) \cup A* (\Delta_1\cap \Delta_3);\K\right)\ \ \ \ \ \forall i>0.
	\end{equation}
	That means that the reduced homology groups of  $\Delta_{A,B}$ does not explicitly depend on $B$.
	
	But we can also write $\Delta_{A,B}$ as 
	\begin{equation}\label{eq.H2}
	\Delta_{A,B} =A*(\Delta_1\cup B*\Delta_3) \cup B*(\Delta_2)
	\end{equation}
		so, by  Lemma \ref{l.m-v} we have
	$$\tilde H_i(\Delta_{A,B};\K)\cong\tilde H_{i-1}\left((\Delta_1\cap \Delta_2) \cup B* (\Delta_2\cap \Delta_3);\K\right)\ \ \ \ \ \forall i>0.$$
	
	To conclude the proof it is enough to take  $U\subseteq [n]$ non-empty such that $U$ is disjoint from  $A,B,A',B'$ and to note that, for equations \eqref{eq.H1} and  \eqref{eq.H2}, we have   
	$$\tilde H_i(\Delta_{A,B};\K)= \tilde H_i(\Delta_{A,U};\K) =\tilde H_i(\Delta_{A',U};\K) =\tilde H_i(\Delta_{A',B'};\K).$$
\end{proof}

\subsection{Betti numbers of toric ideals via simplicial complexes}\hfill\\
Simplicial complexes are involved in the computation of the Betti numbers of a toric algebra in \cite[Theorem 2.1]{BCMP}. Here we recall the notation and we include a short exposition of such result.

Let $G=(V,E)$ be a graph on the vertex set $V=\{x_1, \ldots, x_r\}$ with edges $E=\{e_1, \ldots, e_n \}$.
Consider the semigroups $\mathbb N[V]=\langle x_1,\ldots, x_r \rangle_{\mathbb N}$ and $\mathbb N[E]=\langle e_1, \ldots, e_n\rangle_{\mathbb N}$ and call ${\mathcal V}=[x_1,\ldots, x_r]$ and  
${\mathcal E}=[e_1,\ldots, e_n]$ the standard basis of $\N[V]$ and $\N[E]$. 

Consider the linear map $$\varphi_G: \mathbb N[E]\to \mathbb N[V]$$ induced by $$\varphi_G(e_j)= x_{j_1}+x_{j_2}\ \ \ \ \ \text{where}\ \ \ e_j=\{x_{j_1}, x_{j_2}\}.$$

The incidence matrix $M_G\in \mathbb N^{r,n}$ of the graph $G$ is the matrix associated to the linear map $\varphi_G$ with respect the basis $[e_1, \ldots, e_n]$ and $[x_1,\ldots, x_r]$. 

We have the following commutative square

\[
\begin{array}{lcl}
\mathbb N[E] & \stackrel{\varphi_{G}}{\longrightarrow} & \mathbb N[V]\\
\downarrow^{\pi_{\mathcal E}}\ & &\downarrow^{\pi_{\mathcal V}}\\
\mathbb N^n & \stackrel{M_{G}}{\longrightarrow} & \mathbb N^r\\
\end{array}
\]

where $\pi_{\mathcal V}$ and $\pi_{\mathcal E}$ are the canonical isomorphisms defined by    $$\pi_{\mathcal V}(a_1x_1+a_2x_2+\cdots +a_rx_r)=(a_1,a_2,\ldots, a_r)\  \   \text{and}  
\ \  \pi_{\mathcal E}(b_1e_1+b_2e_2+\cdots +b_ne_n)=(b_1,b_2,\ldots, b_n).$$

For  $s\in \N[V]$ we set for short $\ol s=\pi_{\mathcal V}(s)$; and for  $F\subseteq [n]$ we denote $e_F=\sum_{i\in F}e_i$; we define the simplicial complex
\[
\Delta_{s}^G =\Delta_{\ol s}^G =  
\{  F\subseteq [n] \; | \; 
s -  \varphi_G(e_F) \in Im(\varphi_G) \} .
\]

It is immediate to note that if $s\notin Im(\varphi_G)$ then $\Delta_{s}^G = \emptyset$.

\begin{remark}\label{r. facets}
	Note that given $s\in Im(\varphi_G)$ and $F$  a facet of $\Delta_{s}^G$, then $$s\in \langle\varphi_G(e_i)|\ i\in F \rangle_{\N}. $$ 
	Indeed, $s -  \varphi_G(e_F) \in Im(\varphi_G)$,  so if $s\notin \langle\varphi_G(e_i)|\ i\in F \rangle_{\N}$, then   there exists a $j\in [n]\setminus F$ such that $s -  \varphi_G(e_{F})- \varphi_G(e_j)=s -  \varphi_G(e_{F\cup\{j\}}) \in Im(\varphi_G)$, contradicting the maximality of $F$.
	
	Thus
	  $$s=\varphi_G(\sum_{i\in F}a_ie_i)\ \text{ for some positive integers}\ a_i.$$ 
\end{remark}

We denote by $ \beta_{i,\ol s}(\K[G])$, 
the $i$-th multigraded Betti number of $\K[G]$ in degree $\ol s$. 
Briales, Campillo, Mariju\'{a}n, and Pis\'{o}n  in {\cite[Theorem 2.1]{BCMP}} proved the following result (we write it in the above terminology).	Let $G$ be a finite simple graph, let $s \in \mathbb N[V]$, then 
\begin{equation}\label{l.betti}
	\beta_{j+1, \ol s} (\K[G]) 
	= \dim_{\K} \tilde{H}_j (\Delta_s^G; \K). 
\end{equation}

What follow are preparatory lemmas for the proof of Theorem \ref{t. main 1}.  

\begin{lemma}\label{l.acyc1}
	Let $G=(V,E)$ be a graph and $(e_1,e_2)=(x_1,x_2,x_3)$ be a simple path of $G$.

Let $s=w_1x_1+w_2x_2+\cdots+w_rx_r\in Im(\varphi_G)$ be such that
    $w_1< w_2$. Then 	${\Delta}_{\ol s}$ is acyclic.
\end{lemma}
\begin{proof}
	Let $F$ be a facet of $\Delta_{\ol s}^G$. We claim that  $e_2\in F$. Indeed, by Remark \ref{r. facets}, for some positive integers $a_1, \ldots, a_r$ we have 	$$s = a_1\varphi(e_1)+a_2\varphi(e_2)+\cdots + a_n\varphi(e_n),\ \text{where}\ \ a_j=0 \ \ \text{if}\ \ j\notin F. $$
	Then 
	$$w_1x_1+w_2x_2+\cdots+w_rx_r = a_1(x_1+x_2)+a_2(x_2+x_3)+\cdots$$
	and $x_2$ does not appear in the rest of the sum. Thus
	$w_2=a_1+a_2$ and $a_1\le w_1$. This implies
	$$a_2=w_2-a_1\ge w_2-w_1>0 $$
	so, $e_2$ belongs to the facet $F$.
	Therefore $\Delta_{\ol s}^G$ is a cone and it only has trivial reduced homology groups. 
\end{proof}

\begin{lemma}\label{l.acyc2}
		Let $G=(V,E)$ be a graph. Let $(e_1,e_2,e_3)=(x_1,x_2,x_3,x_4)$ be a path of $G$. 
		Let $s=w_1x_1+w_2x_2+w_3x_3+\cdots+w_rx_r\in Im(\varphi_G)$ be such that ${\Delta}_{\ol s}^G$ is not acyclic. Then  $w_2= w_3$. 
\end{lemma}
\begin{proof}
It follows from Lemma \ref{l.acyc1} since $(e_2,e_1)=(x_3,x_2,x_1)$ and $(e_2,e_3)=(x_2,x_3,x_4)$ are a paths of $G$.
\end{proof}

We are now in the position to prove Theorem \ref{t. main 1}. We recall the statement.

\noindent {\bf Theorem \ref{t. main 1}.} {\it Let $G$ be a simple graph. Let $G/p$ be a even simple path contraction of $G$.  Let $q$ be a path in $G$ containing $p$  and $|q|\ge |p|+2$,  then
$$\beta_i(\K[G])= \beta_i(\K[G/p])\ \ for\ any\ i \ge 1.$$
}
\noindent{\bf Proof of Theorem \ref{t. main 1}.}
Set $G'=G/p$. We prove the statement for $q=(x_1,x_2,x_3,x_4,x_5)=(e_1,e_2,e_3,e_4)$  and $p=(x_2,x_3,x_4)=(e_2,e_3)$. Then, the statement follows by iterating such case.

Let $s=w_1x_1+w_2x_2+w_3x_3+w_4x_4+w_5x_5+ \bar x\in Im(\varphi_G)$ where $\bar x$ is a linear combination of the vertices with index greater than $5.$ 
 
 From Lemma \ref{l.acyc1} if either $w_2> w_1$ or $w_4> w_5$ then the simplicial complex  $\Delta_{s}^{G}$ is acyclic. From Lemma \ref{l.acyc2}, if the coefficients $w_2,w_3,w_4$ are not all the same then the simplicial complex  $\Delta_{s}^{G}$ is acyclic. 
 Thus, we assume $w_2=w_3=w_4=w$, $w\le w_1$   and $w\le w_5.$ Also, we assume $w>0$, since the case $w=0$ is trivial.

Hence the elements in $\varphi^{-1}_G(s)$ must be of type $ae_1+(w-a)e_2+ae_3+(w-a)e_4+\bar e$ for some $0 \le a\le w$  and where $\bar e$ is a linear combination of the edges in $E$ with index greater than 4.   

Thus we note that $$\Delta_{s}^{G}=\{1,3\}*\Delta_1 \cup \{2, 4\}*\Delta_2\cup \{1,2,3,4\}*\Delta_3$$ for some simplicial complex $\Delta_1,\Delta_2,\Delta_3$ on $\{5,\ldots, n\}$.

Let  $\chi:V\to V'$ be the function as in Definition \ref{d. edge contraction}, that is, $\chi(x_i)=x_i$ if $i\neq 2,3,4$ and, say, $\chi(x_2)=\chi(x_3)=\chi(x_4)=y\in [n]\setminus V$.
Then the image of $q$ under $\chi$ is a simple path in $G'$, precisely $(x_1,y,x_5)=(e_i,e_j),$ where $e_i=\{\chi(x_1),\chi(x_2)\} ,e_j=\{\chi(x_4),\chi(x_5)\} \in E'\setminus E.$

Then, in correspondence to the above $s\in Im(\varphi_G)$ we consider
 $s'=w_1x_1+wx_u+w_5x_5+ \bar x\in \N[V']$ and we note that $s'\in Im(\varphi_{G'})$. 
 Indeed, for each $ae_1+(w-a)e_2+ae_3+(w-a)e_4+\bar e\in \varphi^{-1}_{G}(s)$ we have 
 $\varphi_{G'}(ae_i+(w-a)e_j+\bar e)= s',$ with the same $\bar e$ as in $s$.
So we get 
$$\Delta_{s'}^{G'}=\{i\}*\Delta_1 \cup \{j\}*\Delta_2\cup \{i,j\}*\Delta_3.$$
Then, from Proposition \ref{p. m-v case 3} we get
$\beta_i(\K[G])\le \beta_i(\K[G/p])$ for any $i \ge 1,$ and, since Theorem~\ref{t. main 0}, we are done.\qed

\section{Edge contraction of special graphs}\label{s.edge contraction}
As shown in Example \ref{e. contr 3->2}, in a graph $G$ {\it an edge contraction}, i.e., the contraction of a single edge $e$, does not always produces comparable Betti numbers between $\K[G]$ and $\K[G/e]$. However, a particular class of graphs, introduced in the next definition, deserves a further investigation. 

\begin{definition}\label{d. connected by the edge e}
Let $G_1=(V_1,E_1)$ and $G_2=(V_2,E_2)$ be two simple connected graphs such that the vertices sets are  disjoint, i.e., $V_1\cap V_2=\emptyset$.
Let $x\in V_1, y\in V_2$  be two vertices and consider  $e=\{x,y\}$  an edge which connects $G_1$ and $G_2$.
Then, we say that the graph $G=(V_1\cup V_2, E_1\cup E_2\cup \{e\})$ is {\it  connected by the edge} $e$ and we write $G=G_1\stackrel{e}{-}G_2$. 
\end{definition}

We show that the contraction of the edge $e$ in a graph $G$ connected by $e$ preserves the number of minimal generators of $\K[G]$.

\begin{theorem}\label{t.main 2}
Let $G=G_1\stackrel{e}{-}G_2$ be a graph connected by the edge $e$. Then $$\beta_1(\K[G])=\beta_1(\K[G/e]).$$
\end{theorem}
\begin{proof}
Let $w=e||w_1||e||w_2$ be a  primitive even closed walk in $G$ where $w_1$ and $w_2$ are odd closed walk in $G_1$ and $G_2$ respectively. Then $w/e=w_1||w_2$ is a closed even walk in $G/e$. It is easy to check that the binomial corresponding to $w/e$ is $f_{w/e}=e_{w_1^+}e_{w_2^-}-e_{w_1^-}e_{w_2^+}$. Assume it is not primitive. By definition there exists in $I_{G/e}$ a binomial $f_v=e_{v^+}-e_{v^-}$ such that $e_{v^+}|_{e_{w_1^+}e_{w_2^-}}$ and $e_{v^-}|_{e_{w_1^-}e_{w_2^+}}$. We now construct a binomial in order to show that $f_w=e^2e_{w_1^+}e_{w_2^+}-e_{w_1^-}e_{w_2^-}$ is not primitive, a contradiction. If $v$ is an even walk either in $G_1$ or in $G_2$ we are done since $f_v$ will have the required property. Assume $v=v_1||v_2$ where $v_1$ is an odd walk in $G_1$ and $v_2$ in $G_2$, then $f_v=e_{v_1^+}e_{v_2^+}-e_{v_1^-}e_{v_2^-}$. We consider  $\bar v=e||v_1||e||v_2$, that is an even cycle in $G$, thus the corresponding binomial $f_{\bar v}= e^2e_{v_1^+}e_{v_2^-}-e_{v_1^-}e_{v_2^+}$ exhibits the non primitiveness of $f_w$.     
\end{proof}

It is natural to ask then when  the total Betti numbers of a graph $G=G_1\stackrel{e}{-}G_2$ connected by the edge $e$ are preserved by the contraction of $e$. We have the following property.

\begin{proposition}\label{p.bip}
Let $G=G_1\stackrel{e}{-}G_2$ be a graph connected by the edge $e$. If either $G_1$ or $G_2$ is bipartite then $\K[G]\cong \K[G/e]\cong \K[G_1]\bigotimes \K[G_2]$. So in particular $\beta_i(\K[G])=\beta_i(\K[G/e])$ for any $i\ge 0$.     
\end{proposition}
\begin{proof}
Let $G_1$ be a bipartite graph. Then each closed even walk of $G$ containing edges in $G_1$ cannot be primitive. Indeed, an even walk $w$ of $G$ involving edges in $G_1$ and $G_2$ would be of type $w=w_1|| e|| w_2 || e$. So both  $w_1$ and $w_2$ must be even closed walks. Therefore $w$ is certainly not a primitive walk.  

Moreover, the primitive even walks of $G$ are either contained in $G_1$ or in $G_2$. Then $I_G=I_{G_1}+ I_{G_2}$ where the ideals $I_{G_1},I_{G_2}$ are generated by binomials in two different sets of variables. 
\end{proof}

\begin{remark}
    In the case of Proposition \ref{p.bip}, in analogy to \cite[Theorem 2.6]{FHKVT}, the graded Betti numbers of $\K[G]$ are then obtained by those of $\K[G_1]$and $\K[G_2]$ via the K\"{u}nneth formula. See for instance \cite[Theorem 3.4.(b)]{migliore2001bezout} for more details on the K\"{u}nneth formula of graded modules. 
\end{remark}

It is not yet clear when the contraction of $e=\{x,y\}$ in a graph connected by $e$ preserves the total Betti numbers. Also, we observe that the total Betti numbers of $G_1\stackrel{e}{-}G_2$ depends on the vertices $x\in G_1$ and $y\in G_2$, i.e., on how $G_1$ and $G_2$ are connected by $e$. The next example shows three different cases where the total Betti numbers and the edge contraction of the graphs $G_1\stackrel{e}{-}G_2$ depend on the choice of $x$ and $y$. 
\begin{example}\label{e.odd and even contr}
One can check that contracting the edge $e$ in the graph pictured in Figure~\eqref{fig.edgecontr}, the total Betti numbers do not change. The graph $G/e$ is connected by the edge $e'$ The contraction of the path $(e,e')$ pushes down some of the Betti numbers.

\noindent\begin{minipage}{0.3\textwidth}
		\begin{figure}[H]
	\centering
		\begin{tikzpicture}[scale=0.3]
		\coordinate (v1) at (0,0);
		\coordinate (v2) at (4,0);
		\coordinate (v3) at (8,0);
		\coordinate (v4) at (12,0);
\node[fill=white] (a1) at (10,1) {$e$};
		\coordinate (v5) at (16,0);	
\node[fill=white] (a1) at (14,1) {$e'$};
		\coordinate (v6) at (20,0);
		\coordinate (v8) at (2,2);
		\coordinate (v9) at (6,2);
		\coordinate (v10) at (14,-2);	
		\coordinate (v11) at (18,-2);
		\coordinate (v12) at (18,2);	
		\draw (v1)--(v2);
		\draw (v2)--(v3);
		\draw (v1)--(v3);
		\draw (v3)--(v4);
		\draw (v4)--(v5);
		\draw (v5)--(v6);
		\draw (v1)--(v8);
		\draw (v2)--(v8);	
		
		\draw (v2)--(v9);
		\draw (v3)--(v9);
		
		\draw (v5)--(v10);
		\draw (v11)--(v5);
		\draw (v11)--(v10);
		
		\draw (v5)--(v12);
		\draw (v6)--(v12);
		
		\fill[fill=white,draw =black] (v1) circle (0.3);
		\fill[fill=white,draw =black] (v2) circle (0.3);
		\fill[fill=white,draw =black] (v3) circle (0.3);	\fill[fill=white,draw =black] (v4) circle (0.3);
		\fill[fill=white,draw =black] (v5) circle (0.3);
		\fill[fill=white,draw =black] (v6) circle (0.3);
		\fill[fill=white,draw =black] (v8) circle (0.3);
		\fill[fill=white,draw =black] (v9) circle (0.3);
		\fill[fill=white,draw =black] (v10) circle (0.3);
		\fill[fill=white,draw =black] (v11) circle (0.3);
		\fill[fill=white,draw =black] (v12) circle (0.3);
		\end{tikzpicture}
	\caption{ }
	\label{fig.edgecontr}
\end{figure}
\end{minipage}%
\hfill%
\begin{minipage}{0.6\textwidth}\centering
\[
\begin{array}{l|cccc}
i & 0 & 1 & 2 & 3   \\
\hline
    \beta_i (\K[G]) & 1 & 6 & 9 & 4  \\[7pt]
\hline
    \beta_i (\K[G/e]) & 1 & 6 & 9 & 4  \\[7pt]
\hline
     \beta_i (\K[G/(e,e')]) & 1 & 6 & 8 & 3 \\
\end{array}
\]
\end{minipage}		

\vspace{7pt}

\noindent Contracting the edge $e$ in the graph $G$ in Figure \eqref{fig.edgecontr2}, the total Betti numbers go down. The further contraction of  $e'$ does not modify them anymore.

\noindent\begin{minipage}{0.3\textwidth}
		\begin{figure}[H]
	\centering
		\begin{tikzpicture}[scale=0.3]
		\coordinate (v1) at (6,-2);
		\coordinate (v2) at (4,0);
		\coordinate (v3) at (8,0);
		\coordinate (v4) at (12,0);
\node[fill=white] (a1) at (10,1) {$e$};
		\coordinate (v5) at (16,0);	
\node[fill=white] (a1) at (14,1) {$e'$};
		\coordinate (v6) at (20,0);
		\coordinate (v8) at (10,-2);
		\coordinate (v9) at (6,2);
		\coordinate (v10) at (14,-2);	
		\coordinate (v11) at (18,-2);
		\coordinate (v12) at (18,2);	
		\draw (v1)--(v3);
		\draw (v2)--(v3);
		\draw (v1)--(v3);
		\draw (v3)--(v4);
		\draw (v4)--(v5);
		\draw (v5)--(v6);
		\draw (v1)--(v8);
		\draw (v3)--(v8);	
		
		\draw (v2)--(v9);
		\draw (v3)--(v9);
		
		\draw (v5)--(v10);
		\draw (v11)--(v5);
		\draw (v11)--(v10);
		
		\draw (v5)--(v12);
		\draw (v6)--(v12);
		
		\fill[fill=white,draw =black] (v1) circle (0.3);
		\fill[fill=white,draw =black] (v2) circle (0.3);
		\fill[fill=white,draw =black] (v3) circle (0.3);	\fill[fill=white,draw =black] (v4) circle (0.3);
		\fill[fill=white,draw =black] (v5) circle (0.3);
		\fill[fill=white,draw =black] (v6) circle (0.3);
		\fill[fill=white,draw =black] (v8) circle (0.3);
		\fill[fill=white,draw =black] (v9) circle (0.3);
		\fill[fill=white,draw =black] (v10) circle (0.3);
		\fill[fill=white,draw =black] (v11) circle (0.3);
		\fill[fill=white,draw =black] (v12) circle (0.3);
		\end{tikzpicture}
	\caption{ }
	\label{fig.edgecontr2}
\end{figure}
\end{minipage}%
\hfill%
\begin{minipage}{0.6\textwidth}\centering

\[
\begin{array}{l|cccc}
i & 0 & 1 & 2 & 3   \\
\hline
    \beta_i (\K[G]) & 1 & 6 & 9 & 4  \\[7pt]
\hline
    \beta_i (\K[G/e]) & 1 & 6 & 8 & 3  \\[7pt]
\hline
     \beta_i (\K[G/(e,e')]) & 1 & 6 & 8 & 3 \\
\end{array}
\]
\end{minipage}		
		
\vspace{7pt}

The contractions of $e$ and $e'$ in the graph $G$ in Figure \eqref{fig.edgecontr3} do not change the total Betti numbers.

\noindent\begin{minipage}{0.3\textwidth}
		\begin{figure}[H]
	\centering
		\begin{tikzpicture}[scale=0.3]
		\coordinate (v1) at (0,0);
		\coordinate (v2) at (4,0);
		\coordinate (v3) at (8,0);
		\coordinate (v4) at (12,0);
\node[fill=white] (a1) at (10,1) {$e$};
		\coordinate (v5) at (16,0);	
\node[fill=white] (a1) at (14,1) {$e'$};
		\coordinate (v6) at (20,0);
		\coordinate (v8) at (2,2);
		\coordinate (v9) at (6,2);
		\coordinate (v10) at (24,0);	
		\coordinate (v12) at (18,2);	
		\coordinate (v13) at (22,2);
		\draw (v1)--(v2);
		\draw (v2)--(v3);
		\draw (v1)--(v3);
		\draw (v3)--(v4);
		\draw (v4)--(v5);
		\draw (v5)--(v6);
		\draw (v1)--(v8);
		\draw (v2)--(v8);	
		
		\draw (v2)--(v9);
		\draw (v3)--(v9);
		
		\draw (v6)--(v10);
		\draw (v13)--(v6);
		\draw (v13)--(v10);
		
		\draw (v5)--(v12);
		\draw (v6)--(v12);
		
		\fill[fill=white,draw =black] (v1) circle (0.3);
		\fill[fill=white,draw =black] (v2) circle (0.3);
		\fill[fill=white,draw =black] (v3) circle (0.3);	\fill[fill=white,draw =black] (v4) circle (0.3);
		\fill[fill=white,draw =black] (v5) circle (0.3);
		\fill[fill=white,draw =black] (v6) circle (0.3);
		\fill[fill=white,draw =black] (v8) circle (0.3);
		\fill[fill=white,draw =black] (v9) circle (0.3);
		\fill[fill=white,draw =black] (v10) circle (0.3);
		\fill[fill=white,draw =black] (v12) circle (0.3);
		\fill[fill=white,draw =black] (v13) circle (0.3);
		\end{tikzpicture}
	\caption{ }
	\label{fig.edgecontr3}
\end{figure}
\end{minipage}%
\hfill%
\begin{minipage}{0.6\textwidth}\centering
\[
\begin{array}{l|cccc}
i & 0 & 1 & 2 & 3   \\
\hline
    \beta_i (\K[G]) & 1 & 6 & 9 & 4  \\[7pt]
\hline
    \beta_i (\K[G/e]) & 1 & 6 & 9 & 4  \\[7pt]
\hline
     \beta_i (\K[G/(e,e')]) & 1 & 6 & 9 & 4 \\
\end{array}
\]
\end{minipage}				
		
\end{example}

Example \ref{e.odd and even contr} suggests several additional questions. 
\begin{question}\label{q. edge connession Betti numbers}
Let $G$ be a graph connected by the edge $e$. Is $\beta_i(\K[G])\ge\beta_i(\K[G/e])?$
\end{question}
\begin{question}\label{q. edge connession equal Betti numbers}
Let $G$ be a graph connected by the edge $e$. What conditions guarantee $\beta_i(\K[G])=\beta_i(\K[G/e])?$
\end{question}
\begin{question}
Let $G$ be a graph connected by the edge $e$. Is $\K[G]$ Cohen-Macaulay if and only if $\K[G/e]$ is Cohen-Macaulay?
\end{question}

In particular it is interesting to look at a special class of graphs which includes the graph in Figure \eqref{fig.edgecontr3}.

We call  {\it sequence of $n$ triangles}, see Figure \eqref{fig:sequence of triangles},
the graph $T_n=(V,R)$ with vertex set $$V=\{x_1, \ldots, x_{2n+1}\}$$
and edge set
$$E=\left\{\{x_{2i-1},x_{2i}\},\{x_{2i},x_{2i+1}\},\{x_{2i+1},x_{2i-1}\}\ |\ 1\le i\le n \right\}.$$

\begin{figure}[ht]
    \centering
		\begin{tikzpicture}[scale=0.3]
		\coordinate (v1) at (0,0);
		\coordinate (v2) at (4,0);
		\coordinate (v3) at (8,0);
		\coordinate (v5) at (12,0);	
		\coordinate (v6) at (16,0);
		\coordinate (v8) at (2,2);
		\coordinate (v9) at (6,2);
		\coordinate (v10) at (20,0);	
		\coordinate (v12) at (14,2);	
		\coordinate (v13) at (18,2);
		\draw (v1)--(v2);
		\draw (v2)--(v3);
		\draw (v1)--(v3);
		\draw[dotted] (v3)--(v5);
		\draw (v5)--(v6);

		\draw (v1)--(v8);
		\draw (v2)--(v8);	
		
		\draw (v2)--(v9);
		\draw (v3)--(v9);
		
		\draw (v6)--(v10);
		\draw (v13)--(v6);
		\draw (v13)--(v10);
		
		\draw (v5)--(v12);
		\draw (v6)--(v12);
		
		\fill[fill=white,draw =black] (v1) circle (0.3);
		\fill[fill=white,draw =black] (v2) circle (0.3);
		\fill[fill=white,draw =black] (v3) circle (0.3);	
		\fill[fill=white,draw =black] (v5) circle (0.3);
		\fill[fill=white,draw =black] (v6) circle (0.3);
		\fill[fill=white,draw =black] (v8) circle (0.3);
		\fill[fill=white,draw =black] (v9) circle (0.3);
		\fill[fill=white,draw =black] (v10) circle (0.3);
		\fill[fill=white,draw =black] (v12) circle (0.3);
		\fill[fill=white,draw =black] (v13) circle (0.3);
		\end{tikzpicture}
    \caption{The graph $T_n$, a sequence of $n$ triangles.}
    \label{fig:sequence of triangles}
\end{figure}
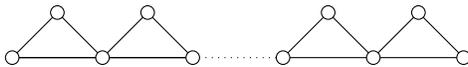

The graph $T_1$ is a cycle of length three, so its toric ideal is the zero ideal.
  
\begin{remark}
Computer experiments suggest the following properties.
\begin{itemize}
    \item[(a)]  The algebra $\K[T_n]$ is Cohen-Macaulay of Krull dimension $n-1$, for any $n\ge 2$.
    \item[(b)] Let $T_m$ and $T_n$ be two graphs sequence of triangles. Let $x\in V(T_m)$ and $y\in V(T_n)$ be vertices of degree $2$ having a neighbor of degree $2$. Set $e=\{x,y\}.$  
Then the algebra $\K[T_n\stackrel{e}{-}T_m]$ is Cohen-Macaulay, for any $n$ and $m$.
    \item[(c)] Let $T_m$ and $T_n$ be two graphs sequence of triangles. Let $x\in V(T_m)$ and $y\in V(T_n)$ be two vertices of degree $2$ having a neighbor of degree $2$. Set $e=\{x,y\}.$  Then
$$\beta_i(\K[T_n\stackrel{e}{-}T_m])=\beta_i(\K[T_n\stackrel{e}{-}T_m/e])=\beta_i(\K[T_{n+m}])=\beta_i(\K[in(T_{n+m})])$$ for any $n$ and $m$, where $in(T_{n+m})$ denotes the initial ideal of $T_{n+m}$ with respect the lexicographic order induced by $e_1>e_2> \cdots> e_{3n}$. 
\end{itemize}

\end{remark}



\begin{thebibliography}{10}

\bibitem{cocoa}
{\sc J.~Abbott, A.~M. Bigatti, and L.~Robbiano}, {\em {CoCoA}: a system for
  doing {C}omputations in {C}ommutative {A}lgebra}.
\newblock Available at \texttt{http://cocoa.dima.unige.it}.

\bibitem{ballard2021}
{\sc L.~Ballard}, {\em Properties of the toric rings of a chordal bipartite
  family of graphs}, in Women in Commutative Algebra, Springer, 2021,
  pp.~11--47. \texttt{https://doi.org/10.1007/978-3-030-91986-3$\_$2 }

\bibitem{BCMP}
{\sc E.~Briales, P.~Pis{\'o}n, A.~Campillo, and C.~Mariju{\'a}n}, {\em
  Combinatorics of syzygies for semigroup algebras}, Collectanea Mathematica,
  (1998), pp.~239--256.

\bibitem{constantinescu2018}
{\sc A.~Constantinescu and E.~Gorla}, {\em Gorenstein liaison for toric ideals
  of graphs}, Journal of Algebra, 502 (2018), pp.~249--261. \texttt{https://doi.org/10.1016/j.jalgebra.2017.12.024}

\bibitem{Adam2022}
{\sc M.~Cummings, S.~Da~Silva, J.~Rajchgot, and A.~Van~Tuyl}, {\em Geometric
  vertex decomposition and liaison for toric ideals of graphs}, arXiv preprint
  arXiv:2207.06391,  (2022).

\bibitem{erey2021}
{\sc N.~Erey and T.~Hibi}, {\em The size of Betti tables of edge ideals arising
  from bipartite graphs}, Proceedings of the American Mathematical Society 150 (2022), pp. 5073-5083. \texttt{https://doi.org/10.1090/proc/16119}

\bibitem{FHKVT}
{\sc G.~Favacchio, J.~Hofscheier, G.~Keiper, and A.~Van~Tuyl}, {\em Splittings
  of toric ideals}, Journal of Algebra, 574 (2021), pp.~409--433. \texttt{https://doi.org/10.1016/j.jalgebra.2021.01.012}

\bibitem{FKVT20}
{\sc G.~Favacchio, G.~Keiper, and A.~Van~Tuyl}, {\em Regularity and
  $h$-polynomials of toric ideals of graphs}, Proceedings of the American
  Mathematical Society, 148  (2020), pp. 4665–4677. \texttt{https://doi.org/10.1090/proc/15126}

\bibitem{robustness2022}
{\sc I.~Garc{\'\i}a-Marco and C.~Tatakis}, {\em On robustness and related
  properties on toric ideals}, Journal of Algebraic Combinatorics,  (2022),
  pp.~1--32. \texttt{https://doi.org/10.1007/s10801-022-01162-x}

\bibitem{ha2019algebraic}
{\sc H.~T. H{\`a}, S.~K. Beyarslan, and A.~O’Keefe}, {\em Algebraic
  properties of toric rings of graphs}, Communications in Algebra, 47 (2019),
  pp.~1--16. \texttt{https://doi.org/10.1080/00927872.2018.1439047}

\bibitem{HH-monomial}
{\sc J.~Herzog and T.~Hibi}, {\em Monomial ideals}, vol.~260, Springer, 2011.  \texttt{https://doi.org/10.1007/978-0-85729-106-6}

\bibitem{herzog2018binomial}
{\sc J.~Herzog, T.~Hibi, and H.~Ohsugi}, {\em Binomial ideals}, vol.~279,
  Springer, 2018. \texttt{https://doi.org/10.1007/978-3-319-95349-6}

\bibitem{hibi2020}
{\sc T.~Hibi, K.~Kimura, K.~Matsuda, and A.~Van~Tuyl}, {\em The regularity and
  $ h $-polynomial of Cameron-Walker graphs}, arXiv preprint arXiv:2003.07416,
  (2020).

\bibitem{K-thesis}
{\sc G.~Keiper}, {\em Toric Ideals of Graphs}, PhD thesis, McMaster University, 2022. \texttt{http://hdl.handle.net/11375/27878}

\bibitem{migliore2001bezout}
{\sc J.~Migliore, U.~Nagel, and C.~Peterson}, {\em Bezout's theorem and
  Cohen-Macaulay modules}, Mathematische Zeitschrift, 237 (2001), pp.~373--394. \texttt{https://doi.org/10.1007/PL00004873}

\bibitem{nandi2019}
{\sc R.~Nandi and R.~Nanduri}, {\em Betti numbers of toric algebras of certain
  bipartite graphs}, Journal of Algebra and its Applications, 18 (2019),
  p.~1950231. \texttt{https://doi.org/10.1142/S0219498819502311}

\bibitem{neves2022}
{\sc J.~Neves}, {\em Eulerian ideals}, Communications in Algebra,  (2022),
  pp.~1--13. \texttt{https://doi.org/10.1080/00927872.2022.2106372}

\bibitem{stanley}
{\sc R.~P. Stanley}, {\em A combinatorial decomposition of acyclic simplicial
  complexes}, Discrete Mathematics, 120 (1993), pp.~175--182. \texttt{https://doi.org/10.1016/0012-365X(93)90574-D} 

\bibitem{V}
{\sc R.~H. Villarreal}, {\em Monomial algebras}, Monographs and Research Notes
  in Mathematics, CRC Press, Boca Raton, FL, second~ed., 2015.
\texttt{https://doi.org/10.1201/b18224}
\end{thebibliography}
\end{document}